\documentclass[11pt]{amsart}
\usepackage{amsmath,amssymb, amscd}
\usepackage{graphicx}

\usepackage{verbatim}

\usepackage{tikz}
\usetikzlibrary{trees}

\def\sqr#1#2{{\vcenter{\hrule height.#2pt
        \hbox{\vrule width.#2pt height#1pt \kern#1pt
                \vrule width.#2pt}
        \hrule height.#2pt}}}

\newtheorem{theorem}{Theorem}[section]
\newtheorem{lemma}[theorem]{Lemma}
\newtheorem{proposition}[theorem]{Proposition}
\newtheorem{corollary}[theorem]{Corollary}

\theoremstyle{definition}
\newtheorem{definition}[theorem]{Definition}
\newtheorem{example}[theorem]{Example}

\newtheorem{question}[theorem]{Question}
\newtheorem{questions}[theorem]{Questions}
\newtheorem{remark}[theorem]{Remark}
\newtheorem{setting}[theorem]{Setting}
\newtheorem{fact}[theorem]{Fact}

\newtheorem{discussion}[theorem]{Discussion}

\DeclareMathOperator{\n}{\mathbf n}
\DeclareMathOperator{\m}{\mathbf m}
\DeclareMathOperator{\p}{\mathbf p}

\DeclareMathOperator{\Spec}{Spec}

\DeclareMathOperator\CIT{CIT}

\DeclareMathOperator\ord{ord}
\DeclareMathOperator\gr{gr}

\DeclareMathOperator\Min{Min}

\DeclareMathOperator\Proj{Proj}
\DeclareMathOperator{\hgt}{ht}

\DeclareMathOperator{\N}{\mathbb N}
\DeclareMathOperator{\Z}{\mathbb Z}

\DeclareMathOperator{\Bl}{Bl}
\DeclareMathOperator{\Rees}{Rees }

\def\alert#1{\smallskip{\hskip\parindent\vrule%
\vbox{\advance\hsize-2\parindent\hrule\smallskip\parindent.4\parindent%
\narrower\noindent#1\smallskip\hrule}\vrule\hfill}\smallskip}

\begin{document}

\title[Finitely supported  $*$-simple complete ideals]
{Finitely supported  $*$-simple complete ideals \\ in a regular local ring }


%\author{William Heinzer and Mee-Kyoung Kim}

\author{William Heinzer}
\address{Department of Mathematics, Purdue University, West
Lafayette, Indiana 47907 U.S.A.}
\email{heinzer@math.purdue.edu}

\author{Mee-Kyoung Kim}
\address{Department of Mathematics, Sungkyunkwan University, Jangangu Suwon
440-746, Korea}
\email{mkkim@skku.edu}

\author{Matthew Toeniskoetter}
\address{Department of Mathematics, Purdue University, West
Lafayette, Indiana 47907 U.S.A.}
\email{mtoenisk@math.purdue.edu}

\date \today


\maketitle
\bigskip

\begin{abstract}  Let $I$ be a finitely supported complete $\m$-primary ideal of a  regular local 
ring  $(R,\m)$.    A  theorem of  Lipman  implies that  $I$ has a unique factorization as 
a $*$-product of special $*$-simple complete ideals with possibly negative exponents for
some of the factors.    The existence of negative exponents occurs  if $\dim R \ge 3$  because 
of the existence of finitely supported $*$-simple ideals that are not special.  
We consider properties  of special $*$-simple complete ideals such as their Rees 
valuations and point basis. 
Let $(R,\m)$ be
a $d$-dimensional equicharacterstic regular local ring with $\m = (x_1, \ldots, x_d)R$. 
%We prove   that many properties  of monomial ideals in a localized polynomial ring over a field also hold
%for  monomial ideals in $R$  with respect to the fixed set $x_1, \ldots, x_d$ of   regular parameters for $R$.   
We define  monomial quadratic transforms of $R$ and 
consider transforms and inverse 
transforms of monomial ideals.   For a large class of monomial ideals  $I$   
that includes complete inverse 
transforms,  we prove that the minimal number of generators of $I$ is completely 
determined by the order of $I$. We give necessary and sufficient conditions for 
the complete inverse transform of a $*$-product of monomial ideals to be the $*$-product 
of the complete inverse transforms of the factors.
This yields  examples  of finitely supported $*$-simple monomial ideals that 
are not special.  We  
 prove that a finitely supported $*$-simple monomial ideal with linearly ordered base points is special $*$-simple.
\end{abstract}

\baselineskip 18 pt

\section{Introduction} \label{c1}

Let $(R,\m)$ be a regular local ring of dimension $d \ge 2$.
In the article
\cite{L},    Lipman    considers the structure of a certain class of 
complete ideals of $R$, the finitely supported complete ideals. 
 He  proves a factorization theorem for the finitely
supported complete ideals that  extends the factorization theory of complete ideals 
in a two-dimensional regular local ring  as developed 
by Zariski  \cite[Appendix 5]{ZS2}.   Heinzer and Kim in \cite{HK1} consider the Rees valuations of 
special $*$-simple complete ideals, and ask whether a $*$-simple complete ideal with linearly ordered 
base points is necessarily special $*$-simple.  We give an affirmative answer to this question 
for the case of monomial ideals.

Other work on finitely supported complete ideals  has been done by  Gately
in \cite{G1} and \cite{G2} and    by Campillo, Gonzalez-Sprinberg and Lejeune-Jalabert in \cite{CGL}.

All rings we consider are assumed to be commutative with an identity element.
We use the concept of complete ideals as defined and discussed in Swanson-Huneke 
\cite[Chapters~5,6,14]{SH}.  We also use a number of concepts considered in Lipman's 
paper \cite{L}. The product of two complete ideals in a two-dimensional regular local 
ring is again complete. This no longer holds in higher dimension, \cite{C} or \cite{Hu}.
To consider the higher dimensional case,  one defines for ideals $I$ and $J$ the $*$-product,
$I*J$ to be the completion of $IJ$.  A complete ideal $I$ in a commutative ring $R$ 
is said to be {\bf $*$-simple} if $I \ne R$ and  if 
$I = J*L$ with ideals $J$ and $L$ in $R$  implies that either $J = R$ or $L = R$.

Another concept used by Zariski in \cite{ZS2} is 
that of the transform of an ideal; the complete transform of an
ideal is used in \cite{L} and \cite{G2}.  

\begin{definition} \label{1.1}
Let $R \subseteq T$ be unique factorization domains (UFDs)  with  $R$
and $T$ having the same field of fractions, and let  $I$ be an ideal of $R$ not contained in 
any proper principal ideal. 
\begin{enumerate}
\item The  {\bf transform } of $I$ in $T$ is the ideal $I^T = a^{-1}IT$,
where $aT$ is the smallest principal ideal in $T$ that contains $IT$.
\item  
The {\bf complete transform} of $I$ in $T$ is the completion $\overline {I^T}$ of $I^T$.
\end{enumerate} 
\end{definition}

A   proper  ideal $I$  in a commutative ring $R$   is {\bf simple} if $I \neq L\cdot H$, for any proper ideals $L$ and $H$.
An element $\alpha \in R$ is said to be {\bf integral over} $I$ 
if $\alpha$ satisfies an equation of the form
$$
\alpha^n + r_1 \alpha^{n-1} + \cdots + r_n = 0, \quad \text{where} \quad r_i \in I^i.
$$
The set of all elements in    $R$  that  are integral over an ideal $I$ forms 
an ideal, denoted by $\overline{I}$ and called the {\bf integral closure} of $I$.
An ideal $I$ is said to be {\bf complete}   (or, {\bf  integrally closed})   if $I = \overline {I}$.

For an ideal $I$ of a local ring $(R,\m)$, the {\bf order} of $I$, denoted $\ord_{R} I$, is $r$ 
if $I \subseteq \m^r$ but $I \nsubseteq \m^{r+1}$.
 If $(R,\m)$ is a regular
local ring, the function that associates to an element $a \in R$,   
the order of the principal ideal $aR$,   defines  a 
discrete rank-one valuation, denoted $\ord_R$  on the field of fractions of $R$.  The associated valuation 
ring (DVR)  is called {\bf the order valuation ring } 
of $R$.

Let $I$ be a nonzero ideal of a Noetherian integral domain $R$. 
The set  of  {\bf Rees valuation rings } of $I$  is  denoted {\bf Rees}$~I$, or by  {\bf Rees}$_RI$
to  also indicate the ring in which $I$ is an ideal. It 
is by  definition the set of DVRs
$$ 
\Big\{\Big(\overline {R\Big[ \frac{I}{a}\Big]}\Big)_{Q}~~ \vert \quad  0 \neq a \in I  \quad 
\text{and} \quad  Q \in \Spec \Big( \overline {R\Big[ \frac{I}{a}\Big]} \Big)\quad 
\text{is of height one with}~~
 I \subset Q \Big\},
$$
where $\overline{\cdot}$ denotes integral closure in the field of fractions. 
The corresponding  discrete valuations with value group $\mathbb Z$ 
are called the {\bf Rees valuations} of $I$.   If  $J \subseteq I$  are ideals of $R$ and $I$ is integral over $J$,  
then  $\Rees J = \Rees I$. 

An ideal $I$ is said to be {\bf normal} if all the powers of $I$ are complete.  
Let $I$ be a normal $\m$-primary ideal of a normal Noetherian  local 
domain $(R,\m)$.  The minimal prime ideals of $\m R[It]$
in the Rees algebra $R[It]$   are in one-to-one correspondence with the Rees valuation rings
of $I$.  The correspondence 
associates to each Rees valuation ring $V$ of $I$ a unique prime 
$P   \in \Min(\m R[It])$ such that $V=R[It]_{P} \cap \mathcal{Q}(R)$.  Properties of the quotient 
ring $R[It]/P$ relate to properties of certain birational extensions of $R$.

We use $\mu (I)$ to denote the minimal number of generators of an ideal $I$.

\section{Preliminaries } \label{c2}

If $R$ is a subring of a valuation domain  $V$   and  $\m_V$   is the  maximal ideal of $V$,  then 
the prime ideal $\m_V \cap  R$    is  called  {\bf the center} of $V$ on $R$.
Let $(R,\m)$ be a Noetherian local domain with field of fractions $\mathcal Q(R)$.
A valuation domain  $(V, \m_V)$ is said to {\bf birationally dominate}  $R$ 
if $R \subseteq V \subseteq \mathcal Q(R)$ and  $\m_V  \cap R = \m$, that is, $\m$ is the 
center of $V$ on $R$.
The valuation domain  $V$ is said
to be a {\bf prime divisor} of $R$ if  $V$ birationally dominates $R$ 
and the transcendence degree of the field $V/\m_V$ 
over $R/\m$ is $\dim R - 1$. If $V$ is a prime divisor of $R$, then $V$ is a DVR
\cite[p. 330]{A}.

The {\bf quadratic dilatation}  or {\bf blowup}  of $\m$ along $V$,  
cf.  \cite[ page~141]{N},  is the
unique local ring on the blowup $\Bl_{\m}(R)$  of $\m$ that is dominated by  $V$.  
The ideal $\m V$ is 
principal and is generated by an element of $\m$.  Let $a \in \m$ be such that $aV = \m V$.  Then
$R[\m/a] \subset V$. Let $Q := \m_V \cap R[\m/a]$. Then $R[\m/a]_Q$ is the  { \bf  quadratic 
transformation   of }
$R$ {\bf along }  $V$.  In the special case where $(R,\m)$ is a $d$-dimensional regular local domain 
we use the following terminology.

\begin{definition}\label{2.1}
Let $d$ be a positive integer and let $(R, \m, k)$ be a $d$-dimensional regular local ring 
with maximal ideal $\m$ and  residue field $k$. 
Let $x\in \m \setminus \m^2$ and let $S_1 := R[\frac{\m}{x}]$. The ring  $S_1$ is a
$d$-dimensional regular ring in the sense that    $\dim S_1 = d$ and each localization of $S_1$ at a 
prime ideal is a regular local ring.  To see this, observe that  $S_1/xS_1$ is isomorphic to a 
polynomial ring in $d-1$ variables over the field $k$, cf. \cite[Corollary~5.5.9]{SH},
 and $S_1[1/x] = R[1/x]$ is a regular ring. Moreover,
$S_1$ is a UFD  since $x$ is a prime element of $S_1$ and
$S_1[1/x] = R[1/x]$ is a UFD, cf. \cite[Theorem~20.2]{M}.
Let $I$  be  an $\m$-primary ideal of $R$ with $r:=\ord_{R}(I)$. Then one has in $S_1$
$$
IS_1=x^rI_1 \quad \text{for some ideal}\quad I_1 \quad \text{of}\quad S_1.
$$
It follows  that either $I_1 = S_1$ or $\hgt I_1 \ge 2$. 
Thus  $I_1$ is  the transform $I^{S_1}$  of $I$ in $S_1$ as
in Definiton~\ref{1.1}. 

Let $\p$ be a prime ideal of $R[\frac{\m}{x}]$ with $\m \subseteq \p$. 
      The local ring 
$$R_1:~= ~R[\frac{\m}{x}]_{\p}  ~ = (S_1)_{\p}
$$
 is called a {\bf local quadratic transform} of $R$;  the ideal
$I_1R_1$ is the  transform of $I$ in $R_1$ as in Definition~\ref{1.1}.  
\end{definition}

We follow the notation of \cite{L} and  refer to regular local rings of dimension at
least two 
as   {\bf points}.
A point $T$ is said to be {\bf infinitely near} to a point $R$,  in symbols, $R~ \prec~ T$, 
if there is a finite sequence of local quadratic transformations 
\begin{equation} \label{e1} 
R=:R_0 ~\subset~ R_1 ~\subset~ R_2~ \subset \cdots \subset~ R_n=T \quad (n \geq 0),
\end{equation}
where $R_{i+1}$ is a local quadratic transform of $R_i$ for $i=0,1, \ldots, n-1$.
If such a sequence of local quadratic transforms as in (\ref{e1}) exists, then it is unique and 
it is
called the {\bf quadratic sequence} from $R$ to $T$ \cite[Definition~1.6]{L}.   The set of points $T$ infinitely 
near to $R$ such that $T$ is a local quadratic transform of $R$ is called the {\bf first neighborhood } of $R$.
If $T$ is in the first neighborhood of $R$, a point in the first neighborhood of $T$ is said to be in 
the {\bf second neighborhood} of $R$.  Similar terminology is used for each positive integer $n$.

\begin{remark} \label{2.111} Let $(R,\m)$ be a regular local ring with $\dim R \ge 2$. 
As noted in \cite[Proposition~ 1.7]{L}, there is a one-to-one correspondence between the 
points $T$ infinitely near to $R$ and the prime divisors $V$ of $R$. This correspondence
is defined by associating with $T$ the order valuation ring $V$ of $T$.   Since $V$ is the unique  local quadratic transform of $T$
of dimension one,  the local quadratic sequence in (\ref{e1})  extends to give (\ref{e11}):
\begin{equation} \label{e11} 
R=:R_0 ~\subset~ R_1 ~\subset~ R_2~ \subset \cdots \subset~ R_n=T   ~\subset ~V.
\end{equation}
The one-to-one correspondence between the 
points $T$ infinitely near to $R$ and the prime divisors $V$ of $R$ implies that
$T$ is the unique point  infinitely near to $R$ for which the order valuation ring of $T$ is $V$.
However, if  $\dim R > 2$, then there  often exist 
regular local rings $S$ with $S \ne T$ such that $S$ birationally dominates $R$ 
and the order valuation ring of $S$ is $V$,   cf.  \cite[Example~2.4]{HK1}.    Indeed,  Example~2.6 of \cite{HK1}  
 demonstrates the existence of a prime divisor $V$ 
for a $3$-dimensional RLR  $(R, \m, k)$ for which there 
exist infinitely many distinct $3$-dimensional RLRs that 
birationally dominate $R$, and  have $V$ as their  order valuation ring.
\end{remark} 

\begin{remark} \label{2.113}
Let $(R,\m)$ be a $d$-dimensional RLR with $d \ge 2$ and let $V$ be the order
valuation ring of $R$. Let $(S,\n)$ be a $d$-dimensional RLR that is a birational
extension of $R$. Then
\begin{enumerate}
\item $S$ dominates $R$.
\item If $V$ dominates $S$, then $R = S$.
\item  Thus $R$ is the unique $d$-dimensional RLR having order valuation ring $V$ among
the regular local rings birational over $R$.
\end{enumerate}
\end{remark}

\begin{proof}  For item (1), let $P := \n \cap R$. Then $R_P \subseteq S$. If $P \ne \m$, 
then $\dim R_P = n < d$. Since every  birational extension of an $n$-dimensional Noetherian
domain has dimension at most $n$, we must have $\dim S \le n$, a contradiction. Thus $S$
dominates $R$. Item (2) follows from \cite[Theorem~2.1]{Sa2}.  In more detail, if $V$ dominates
$S$, then $R/\m = S/\n$ and the elements in a minimal generating set for $\m$ are part of 
a minimal generating set for $\n$. Hence we have $\m S = \n$. By Zariski's Main Theorem as in
\cite[(37.4)]{N}, it follows that $R = S$. Item (3) follows from item (2).
\end{proof}

\begin{definition} \label{2.12} 
A {\bf base point} of a nonzero ideal $I \subset R$ is a point $T$  infinitely near 
 to  $R$ such 
that $I^{T} \neq T$. The set of base points of $I$ is denoted by
$$
\mathcal{BP}(I)=\{~ T~\vert  ~T \text{ is a point such that}~~R \prec T ~\text{and}~\ord_T (I^T)\neq 0~\}.
$$
The {\bf point basis} of a nonzero ideal $I \subset R$ is the family of nonnegative integers
$$
\mathcal B(I)=\{~\ord_T (I^T) ~\vert~ R~ \prec~ T ~\}.
$$
The nonzero ideal  $I$ is said to be {\bf finitely supported} if  
$I$ has only finitely many base points.
\end{definition}

\begin{definition} \label{2.13}
Let $R \prec T$ be points such that $\dim R=\dim T$. Lipman proves in \cite[Proposition~2.1]{L}
the existence of a unique complete ideal $P_{RT}$ in $R$ such that for every point $A$
with $R \prec A$,  the 
complete transform 
$$
\overline{(P_{RT})^A} ~ \text{ is }  ~
\begin{cases}
\text{ a $*$-simple ideal if  } ~ A \prec T, \\ \text{ the ring } A \text{ otherwise.} 
\end{cases}  
$$
The ideal $P_{RT}$ of $R$ is said to be  a {\bf special $*$-simple complete ideal}. 

In the case where $R \prec T$ and  $\dim R = \dim T$, we say that 
the order valuation ring of $T$
is a {\bf special prime divisor} of $R$.
\end{definition} 

\begin{remark} \label{2.131}
With notation at in Definition~\ref{2.13},  a prime divisor $V$ of $R$ is 
special if and only if the unique point $T$ with $R \prec T$ such that the
order valuation ring of $T$ is $V$ has $\dim T = \dim R$. Let $\dim R = d$. 
If  $V$ is a special prime divisor of $R$, then the residue field of $V$ is 
a pure transcendental extension of degree $d - 1$ of the residue field $T/\m(T)$
of $T$, and $T/\m(T)$ is  
a finite algebraic
extension of $R/\m$.  If the residue field $R/\m$ of $R$ is 
algebraically closed and $V$ is a special prime divisor of $R$, then the 
residue field of $V$ is a pure transcendental extension of $R/\m$ of transcendence
degree $d - 1$.  
It would be interesting to  identify   and describe in other ways   the special prime divisors of $R$ 
among the set of all prime divisors of $R$.
\end{remark}

\begin{remark} \label{lipman1}
Let $R = \alpha$ be a $d$-dimensional regular local ring with $d \ge 2$.
 Lipman in \cite[Theorem~2.5]{L}
proves that for every finitely supported complete ideal $I$ of $R$ there exists a 
unique family of integers 
$$
(n_\beta) ~ = ~ (n_\beta(I))_{\beta \succ \alpha, ~ \dim \beta = \dim \alpha}
$$
such that $n_\beta = 0$ for all but finitely many $\beta$  and such that 
\begin{equation} \label{es30}
\Big( \prod_{n_\delta < 0} P_{\alpha\delta}^{-n_\delta} \Big)*I ~=~ 
\prod_{n_\gamma > 0}P_{\alpha\gamma}^{n_\gamma}
\end{equation}
where $P_{\alpha\beta}$ is the special $*$-simple ideal associated with $\alpha \prec \beta$
and the products are $*$-products.  The product 
on the left is  over all $\delta \succ \alpha$ such that $n_\delta < 0$ and the product on 
the right is over all $\gamma \succ \alpha$ such that $n_\gamma > 0$.

A question of interest is for which finitely supported complete ideals $I$ the unique factorization 
of $I$ given in Equation~\ref{es30}  involves special $*$-simple ideals $P_{\alpha \delta}$ with  $n_{\delta} < 0$.  If 
$\dim R = 2$ there are no negative exponents and every $*$-simple complete ideal is special $*$-simple.
\end{remark}

Lipman gives the following example to illustrate this decomposition.

\begin{example} \label{s3.1} Let $k$ be a field and let $\alpha = R =  k[[x,y,z]]$ be the
formal power series ring in the 3 variables $x,y,z$ over $k$.  Let 
$$
\beta_x = R[\frac{y}{x}, \frac{z}{x}]_{(x, y/x, z/x)},  \quad 
\beta_y = R[\frac{x}{y}, \frac{z}{y}]_{(y, x/y, z/y)},   \quad 
\beta_z = R[\frac{x}{z}, \frac{y}{z}]_{(z, x/z, y/z)} 
$$ 
be the local quadratic transformations of $R$ in the $x,y,z$ directions. The associated 
special $*$-simple ideals are
$$
P_{\alpha \beta_x} = (x^2, y, z)R, ~ \quad P_{\alpha \beta_y} = (x,y^2, z)R, ~ \quad
P_{\alpha \beta_z} = (x,y,z^2)R.
$$
The equation
\begin{equation} \label{es1}
(x,y,z)(x^3, y^3, z^3, xy, xz, yz) ~= ~ P_{\alpha \beta_x} P_{\alpha \beta_y} P_{\alpha \beta_z}
\end{equation}
represents the factorization of the finitely supported ideal $I = (x^3, y^3, z^3, xy, xz, yz)R$
as a product of special $*$-simple ideals. Here $P_{\alpha \alpha} = (x,y,z)R$. The base points
of $I$ are 
$\mathcal {BP}(I) = \{\alpha, \beta_x, \beta_y, \beta_z\}$ and the point basis of $I$ is
$\mathcal B(I) = \{2,1,1,1\}$. Equation~(\ref{es1}) represents the following equality of 
point bases
$$
\mathcal B(P_{\alpha\alpha})~ + ~ \mathcal B(I) ~=~ \mathcal B(P_{\alpha \beta_x})
~+ ~ \mathcal B( P_{\alpha \beta_y}) ~+~ \mathcal B( P_{\alpha \beta_z}).
$$
Each of $P_{\alpha \beta_x},  P_{\alpha \beta_y},  P_{\alpha \beta_z}$ has a unique Rees
valuation. Their product has in addition the order valuation of $\alpha$ as a Rees valuation. 
  Lipman's unique factorization theorem given in Equation~\ref{es30}  along with 
Fact~\ref{order1} imply that  the ideal  $I$ is $*$-simple.
To see this,  suppose by way of contradiction that  $I$ has a non-trivial $*$-factorization $I = L * W$.  
Then $L$ and $W$  have order $1$, and hence by Fact~\ref{order1} are  special $*$-simple ideals.  
But this 
contradicts  Lipman's   unique factorization theorem for  $I$ as a product of special $*$-simple ideals as stated in Equation~\ref{es30}.  
We conclude that $I$ is $*$-simple.
\end{example}

\begin{remark}  \label{factprop}  (1)  The finite set $\mathcal {BP}(I)$ of base points of the finitely supported
complete ideal $I$ of Remark~\ref{lipman1} is a partially ordered set  with respect  to inclusion.
This partially ordered set is a rooted tree with unique minimal element $\alpha = R$. 
For each base point $\beta$ of $I$ there is a unique finite sequence of base points
$\alpha \prec \gamma_1 \prec \cdots  \prec \gamma_n = \beta$ of $I$, where $\gamma_1$ is a  local 
quadratic transform of $\alpha$,  and each $\gamma_{i+1}$ is a  local quadratic transform 
of $\gamma_i$. Thus there exists a unique path from $\alpha$ to each base point $\beta$. 
A base point $\beta$ of $I$ is a {\bf maximal base point}   of $I$ if   
$\beta$ is a maximal element in the 
partially ordered set $\mathcal {BP}(I)$,  that is, if $\beta \prec \gamma$ with $\gamma$ a 
base point of $I$,  then  $\beta = \gamma$.  In the unique $*$-factorization of 
$I$ given in Equation~\ref{es30}, the integer $n_{\beta}$ associated to 
each maximal  base point $\beta$  of $I$ is equal to the point basis of $I$ at $\beta$ and hence  is positive.\\
(2)  Let $I$ be a finitely supported ideal of a $d$-dimensional  regular local ring $(R,\m)$.  For each $x \in \m \setminus \m^2$,
Corollary 1.22 of \cite{L}    implies that  
the transform $I^{S}$ of $I$ in the ring $S = R[\m/x]$ is either the  ring  $S$ or an ideal 
of height $d$ in $S$.
 \end{remark}

 \begin{fact}  \label{order1}
Let $J$ be a finitely supported complete  $\m$-primary  ideal of a  $d$-dimensional  regular local ring $(R,\m)$.
If $J$ has order one, then $J$ is  a special $*$-simple ideal  and has the form 
$J = (x_1, \ldots, x_{d-1},~ x_d^{n})R$,    for  some positive integer $n$  and regular system of parameters $x_1, \ldots, x_d$ 
for $R$.
\begin{proof}
Since $J$ has order one,  either $J = \m$  or there exists a positive  integer $e  <  d$ and a regular system of 
parameters $x_1, \ldots, x_d$  for $R$ such that $J = (x_1, \ldots, x_e) R + J'$,  where  the image of $J'$ in the 
regular local ring $\frac{R}{(x_1, \ldots, x_e)R}$ has order $n \ge 2$  and  where also $\ord_{R} J' = n$.
Since $J$ is $\m$-primary, $J'$ contains some power of $x_d$.
Assume that $J \ne \m$,  and let $S = R [\frac{\m}{x_d}]$.
Then the transform $J^{S}$ of $J$ in $S$ is
$$
J^{S ~} =  ~\big( \frac{x_1}{x_d}, \ldots, \frac{x_{e}}{x_d} \big)S ~ +  ~ x_d^{n-1} J'^{S}.
$$
This equality and the fact that $J'^{S}$ contains some power of $x_d$ 
implies that $J^{S}$ has radical $(\frac{x_1}{x_d}, \ldots, \frac{x_e}{x_d}, ~ x_d)S$.
Hence the unique minimal prime of $J^{S}$    is $(\frac{x_1}{x_d}, \ldots, \frac{x_e}{x_d}, x_d)S$.
By item~2 of Remark~\ref{factprop},  the ideal    $J^{S}$   has height $d$.  Therefore
$e = d-1$.  
Since $ \frac{R }{ (x_1, \ldots, x_{d-1}) R}$ is a DVR  whose maximal ideal is generated by the image of $x_d$, 
we have $J = (x_1, \ldots, x_{d-1}, ~ x_d^n)R$.
A simple induction proof yields  that $J$ is special $*$-simple.
\end{proof}
\end{fact}

%%%%%%%%%%%%%%%%%%%%%%%%%%%%%%%--NEW SECTION
\section{Point basis and change of direction }\label{c3}

\begin{setting} \label{7.1} Let $(R,\m)$ be a regular local ring, and 
consider the sequence 
\begin{equation} \label{e2}
R=:R_0 ~\subset~ R_1 ~\subset~ R_2~ \subset \cdots \subset~ R_n=T \quad (n \geq 2),
\end{equation}
where $R_{i+1}$ is a local quadratic transform of $R_i$ for $i=0,1, \ldots, n-1$.   
\end{setting}

\begin{remark} \label{7.1.5}  Let notation be as  in Setting~\ref{7.1}
and assume $\dim R = \dim T$. 
\begin{enumerate}
\item
If  $\dim R = 2$, then the special $*$-simple complete ideal $P_{RT}$ has
a unique Rees valuation $\ord_T$.  
\
\item
In the higher dimensional case, the  special $*$-simple complete ideal $P_{RT}$ has
$\ord_T$ as a Rees valuation and often also has other Rees valuations. We observe in
\cite[Proposition~ 6.4]{HK1} that the other Rees valuations of $P_{RT}$ are in the set
$\{\ord_{R_i} \}_{i=0}^{n-1}$.
\item
The residue field $R_n/\m_n$ of $R_n$ is a finite algebraic
extension of the residue field $R_0/\m_0$ of $R_0$. 
\begin{enumerate}
\item
If $R_0/\m_0 = R_n/\m_n$,  then  
the other Rees valuations of $P_{RT}$ are in the set
$\{\ord_{R_i} \}_{i=0}^{n-2}$, cf. \cite[Corollary~ 4.11]{HK1}.
\item
If  $R_0/\m_0 \subsetneq  R_n/\m_n$, then 
 $\ord_{R_{n-1}}$ may be a Rees valuation of  $P_{RT}$, cf. \cite[Example~6.11]{HK1}.
\end{enumerate}
\end{enumerate}

\end{remark}

\begin{definition} \label{7.11d} We say {\bf there is no change of direction} for 
the local quadratic sequence $R_0$ to $R_n$ in Equation~(\ref{e2}) if there exists 
an element $x \in \m_0$ that is part of a minimal generating set of $\m_n$.  We
say {\bf there is a change of direction} between $R_0$ and $R_n$ if $\m_0 \subseteq \m_n^2$. 
\end{definition} 

\begin{remark}\label{7.12} Concerning  change of direction for 
the local quadratic sequence $R_0$ to $R_n$ of Equation~(\ref{e2}), 
the following statements are equivalent:
\begin{enumerate}
\item  There exists 
 $x \in \m_0$ that is part of a minimal generating set of $\m_n$.
\item $\m_0 \nsubseteq \m_n^2$.
\item  $\m_0 \nsubseteq \m_i^2$  for all $i=0,1,\ldots, n$.
\item  There exists $x \in \m_0$ such that 
$\ord_{R_i}(\m_0)=\ord_{R_i}(x)=1$ for all $i=0,1,\ldots, n$.
\item   There exists $x \in \m_0$ such that 
$\m_i R_{i+1}=x R_{i+1}$ for all $i=0,1,\ldots, n-1$.
\end{enumerate}
\end{remark}

\begin{remark} \label{7.2} With notation as in Setting~\ref{7.1},  assume that 
$\dim R = \dim R_n$.  Let $I = P_{R_0R_n}$. 
By \cite[Corollary~2.2]{L}, the transform $I^{R_i} = P_{R_iR_n}$ for
all $i$ with $0 \le i \le n$. By \cite[Proposition~4.6]{HK1}, we have
$\Rees_{R_i} I^{R_i} \subseteq \Rees I$.  Thus for each $i$ with $0 \le i \le n$, we have
$$
\Rees_{R_i} P_{R_iR_n} ~ = ~ \Rees_{R_i} I^{R_i} ~ \subseteq ~ \Rees I,
$$
and the number of Rees valuations of $I$ is greater than or equal to the
number of Rees valuations of $P_{R_iRn}$. 
\end{remark}

\begin{lemma}\label{7.4}
With notation as in Setting~\ref{7.1} assume that $d:=\dim R=\dim R_1$ and $R/\m=R_1/\m_1$.
Let $J \subseteq \m$ be a nonzero ideal of $R$, and let $J^{R_1}$ denote the transform of $J$ in $R_1$.
Then $  \ord_R(J)\geq \ord_{R_1}(J^{R_1})$.
\end{lemma}

\begin{proof}
By appropriate choice of a regular system of parameters $x, y, \ldots, z$ for $R$,  
we may assume that $R_1$ is a localization of  $S = R [\frac{\m}{x}]$  at the 
maximal ideal  $(x, \frac{y}{x}, \ldots, \frac{z}{x})S$.
To show $\ord_{R} (J) \geq \ord_{R_1} (J^{R_1})$,
it suffices to show that for $f \in J$ such that $\ord_{R} (f) = \ord_{R} (J) = r$, it follows that $\ord_{R_1} (\frac{f}{x^r}) \leq r$.
To prove this, we identify more explicitly certain isomorphisms connected with the transform.
Consider the Rees algebra 
$R [\m t] = \bigoplus_{n \ge 0}   \m^nt^n$ 
and the associated graded ring $\gr_{\m} (R) := R [\m t] / \m R [\m t] \cong k [X, Y, \ldots, Z]$, 
a polynomial ring in $d$ variables over the field $k := R / \m$.
For $f \in R$ with $\ord_{R} f = r$, let $f^*$ denote the leading form of $f$ in $\gr_{\m} (R)$,
that is, $f^{*}$ is the image of $f t^{r} \in \m^rt^r \subset  R [\m t]$ 
in its quotient $\gr_{\m} (R)$.

As a $\Z$-graded ring,  we have  $ R [\m t] [\frac{1}{x t}]   = S [x t, \frac{1}{x t}]$ is a Laurent polynomial ring in $xt$ over $S$,
and  $\frac{f}{x^r} = \frac{f t^r}{(x t)^r}$ is in $S$.   Since  $\m S = x S$, we   obtain by permutability of localization 
and residue class formation 
\begin{equation}  \label{rees6}
S \big[x t, \frac{1}{x t} \big] / x S \big[x t, \frac{1}{x t} \big] ~ =   ~(S / x S) \big[x t, \frac{1}{x t} \big] ~ \cong   ~\gr_{\m} (R) \big[ \frac{1}{X } \big].
\end{equation}
The isomorphism in Equation~\ref{rees6} identifies  $\frac{1}{xt}$  with $\frac{1}{X}$, and  
the ring  $S / x S$ with the polynomial ring in $d-1$ variables  $A := k [\frac{Y}{X}, \ldots, \frac{Z}{X}]$ 
over the field $k$. 
Equation~\ref{rees6}  also identifies  the image  $w$  of $\frac{f}{x^r}$ in $S / x S$
with $ \frac{f^*}{X^r}$,  a polynomial in $A$ of degree $\le r$.  
 Since $R_1$ is the localization of $S$ at the maximal ideal $(x, \frac{y}{x}, \ldots, \frac{z}{x})S$,   the $d-1$-dimensional regular local ring 
$R_1 / x R_1$ is isomorphic to  the localized polynomial ring $B := k[\frac{Y}{X}, \ldots, \frac{Z}{X}]_{(\frac{Y}{X}, \ldots, \frac{Z}{X})}$.
For a polynomial $g \in A$, we have $\deg g \le \ord_B(g)$.  
Therefore $\ord_{B} (w) \le r$.
Since $\ord_{R_1} (\frac{f}{x^r}) \le \ord_{R_1 / x R_1}(w) = \ord_B(w)$, we conclude that $\ord_{R_1} (\frac{f}{x^r}) \le r$.
\end{proof}

We observe in Remark~\ref{7.4r} that Lemma~\ref{7.4} holds even without the assumption that $R / \m = R_1 / \m_1$.

\begin{remark}\label{7.4r}
Assume that $R \subset R_1$ are $d$-dimensional regular local rings and that 
$R_1$ is a local quadratic transform of $R$ with $R/\m \subsetneq  R_1/\m_1$.  We  may 
assume that $R_1$ is a localization of $S = R[\frac{\m}{x}]$ with respect to a maximal ideal of 
$S$ that contains $x$.  As in the proof of Lemma~\ref{7.4}, we obtain a polynomial ring $A$ in
$d-1$ variables over $k = R/\m$  such that $R_1/xR_1$ is isomorphic to $A_N$,
where $N$ is a maximal ideal of $A$.  Since $R/\m \subsetneq  R_1/\m_1$, the field  $A/N$ 
is a proper finite algebraic extension of $k$.   

For $A$ a polynomial ring in $n$ variables over a field $F$ and $N$ a maximal ideal of $A$, we prove by induction on $n$ the following statement:
\begin{itemize}
\item For each polynomial $h \in A$, we have $\ord_{N} (h) \le \deg h$.
\end{itemize}
This is clear for $n = 1$ by factoring $h$ as a product of irreducible polynomials.
If $n > 1$ and $A = F [t_1, \ldots, t_n]$,  then $P = N \cap F [t_n]$ is a maximal ideal of $F [t_n]$ by Hilbert's Nullstellensatz. 
Hence  $P$  is generated by  an irreducible polynomial $p (t_n) \in N \cap F [t_n]$.
Let $B = A / (N \cap F [t_n]) \cong E [t_1, \ldots, t_{n-1}]$, where $E = F [t_n] / p (t_n) F [t_n]$ is a finite algebraic field extension of $F$.
Let $\overline{h}$ denote the image of $h$ in $B$ and let $\overline{N}$ denote the image of $N$ in $B$.
By the induction hypothesis, $\ord_{\overline{N}} (\overline{h}) \le \deg \overline{h}$.
Since $\deg \overline{h} \le \deg h$ and $\ord_{\overline{N}} (\overline{h}) \ge \ord_{N} (h)$, the claim follows.

This implies for an ideal $J$ in $R$ that $\ord_R J \ge \ord_{R_1} J^{R_1}$ also in 
the case where $R/\m \subsetneq R_1/\m_1$.
\end{remark}

\begin{proposition}\label{7.3}  Assume notation as in Setting~\ref{7.1} with 
$\dim R = \dim R_n$ and $R_0/\m_0 = R_n/\m_n$.   Let $I = P_{R_0R_n}$.
 Then:
\begin{enumerate}
\item
 $ \ord_{R_n}(I^{R_n})=\ord_{R_{n-1}}(I^{R_{n-1}})=1$
\item 
$ \ord_{R_i}(I^{R_i}) \geq \ord_{R_{i+1}}(I^{R_{i+1}})$ for $i$ with $0 \leq i \leq n-1$.
\item
$\mathcal B(I)=\{r_0,~r_1,~ \ldots, r_{n-2},~ 1,~ 1\}$
is a decreasing sequence, where $r_i:=\ord_{R_i}(I^{R_i})$ for $i$ with $0 \leq i \leq n$. 
\end{enumerate}
\end{proposition}

\begin{proof}
(1): Since $I^{R_{n-1}}=P_{R_{n-1}R_{n}}$,  we have $\mathcal B(I^{R_{n-1}})=\{1,~1\}$, by \cite[Remark~6.5]{HK1}.
Item (2) follows from Lemma~\ref{7.4}. Item (3) is immediate from items~(1) and (2). 
\end{proof}

\begin{theorem} \label{3.1}
Let $(R,\m)$ be a universally catenary  analytically 
unramified Noetherian local domain with $\dim R = d$, and 
let $V$ be a prime divisor of $R$ centered on $\m$. Let $I \subseteq \m$ be 
an ideal of $R$. The following are equivalent
\begin{enumerate}
\item $V \in \Rees I$.
\item There exist elements $b_1, \ldots, b_d$ in $I$ such that 
$b_1V = \cdots= b_dV = IV$
and the images of  $~\frac{b_2}{b_1}, \ldots, \frac{b_d}{b_1}$ in the residue 
field $k_v$ of $V$ are algebraically
independent over $R/\m$.
\item If $I = (a_1, \ldots, a_n)R$, then there exist 
elements $b_1, \ldots, b_d$ in $\{a_i\}_{i=1}^n$ such that 
$b_1V = \cdots= b_dV = IV$
and the images of $~\frac{b_2}{b_1}, \ldots, \frac{b_d}{b_1}$ in the residue 
field $k_v$ of $V$ are algebraically
independent over $R/\m$.
\end{enumerate}
Thus  if $I = (a_1, a_2, \ldots, a_d)R$, then $V \in \Rees I  \iff  a_1V = a_2V =\cdots= a_dV$
and the images of $\frac{a_2}{a_1}, \ldots, \frac{a_d}{a_1}$ in 
$k_v$ are algebraically
independent over $R/\m$.
\end{theorem}

\begin{proof} $(1) \Rightarrow  (2)$: If $V \in \Rees I$, then by definition
$V$ is a localization of the integral closure of the 
ring   $R[I/b]$ for some nonzero $b \in I$.  We have $I/b  \subset V$ and so
$IV = bV$. 
Since 
$V$  has center $\m$ on $R$ and is centered on a prime ideal of height one of 
the integral closure of $R[I/b]$,  
the dimension formula \cite[Theorem~B.5.1, p.~403]{SH}, or \cite[p.~119]{M}, implies that
there exist elements $b_1 := b, b_2, \ldots, b_d$ in $I$ such that
$b_1V = \cdots= b_dV = IV$
and the images of  $~\frac{b_2}{b_1}, \ldots, \frac{b_d}{b_1}$ in the residue 
field $k_v$ of $V$ are algebraically
independent over $R/\m$. \\
$(2) \Rightarrow (3)$:  If $I = (a_1, \ldots, a_n)R$, we may assume that $a_1V = IV$. 
It follows that $R[I/a_1] \subseteq V$ and $V$ is centered on a height-one prime 
of the integral closure of $R[I/a_1]$. We have 
$R[I/a_1] = R[\frac{a_2}{a_1}, \ldots, \frac{a_n}{a_1}]$.
By the dimension formula, there exist elements $b_2, \ldots, b_d$ in $\{a_i\}_{i=2}^n$
such that 
$b_1V = \cdots= b_dV = IV$
and the images of  $~\frac{b_2}{b_1}, \ldots, \frac{b_d}{b_1}$ in the residue 
field $k_v$ of $V$ are algebraically
independent over $R/\m$. \\
$(3) \Rightarrow (1)$: We have $R[\frac{b_2}{b_1},\ldots,  \frac{b_d}{b_1}] \subseteq R[I/b_1]$.
Let $A$ denote the integral closure of
 $R[I/b_1]$ and let $\p$
denote the center of $V$ on $A$. 
We have $\p \cap R = \m$, and
$$
R/\m ~ \subseteq ~A/\p ~\subseteq V/\m_v ~= k_v.
$$ 
Since the images  of $\frac{b_2}{b_1}, \ldots, \frac{b_d}{b_1}$
 in $A/\p  \subseteq k_v$ are algebraically independent over $R/\m$, 
the dimension formula implies that   $\hgt \p = 1$. 
Thus $A_{\p}$ is a DVR that is birationally dominated by $V$.  Hence  $A_{\p} = V$, and 
$V \in \Rees I$.

The last sentence is immediate from the equivalence of items 1, 2, and 3.
\end{proof} 

\begin{proposition}\label{contracted}
% Assume the notation of Lemma~\ref{5.14} and  let $K$ be an ideal of $R$ 
%that is contracted from $^{x}S = R [\frac{y}{x}, \ldots, \frac{z}{x}]$.
%For example, $K$ may be the complete inverse transform of an $\m_1$-primary ideal of $R_1$.

Let $(R, \m)$ be a $d$-dimensional regular local ring with $d \ge 2$, let $x \in \m \setminus \m^2$, and let $S = R [\frac{\m}{x}]$.
Let $K$ be an ideal of $R$ that is contracted from $S$, and let $x f \in K$, where $f \in R$.
Then
\begin{enumerate}
\item
$g f \in K$ for each $g \in \m$.
\item
If $\ord_{R} (x f) = \ord_{R} (K)$, then the order valuation $\ord_R$ 
is a Rees valuation of $K$.
\end{enumerate}
\end{proposition}

\begin{proof}
For  item~1, notice that $\frac{g}{x} \in S$  implies  
$x \frac{g}{x} f = g f \in K S \cap R = K$. 
For item~2, let $V$ denote the order valuation ring of $R$.  The assumption that 
$xf \in K$ and $\ord_R xf = \ord_R(K)$ implies 
by item~1   that also  $ y f, \ldots,$ and $ z f$ are in $K$.   We have $K V = xfV = yfV = \cdots = zfV$.     
In the residue field $V / \m_V$ of $V$,  the $d - 1$ elements  
$\overline{\frac{y f}{x f}} = \overline{\frac{y}{x}}, \ldots, \overline{\frac{z f}{x f}} = \overline{\frac{z}{x}}$
 form a transcendence basis for  $V/\m_V$ as an extension field of  $R / \m$.
By Theorem~\ref{3.1}, $V \in \Rees  K$.
\end{proof}

Proposition~\ref{reestrans} is useful for relating the Rees valuations
 of the transform of an ideal to the Rees valuations of the ideal.

\begin{proposition} \label{reestrans}
Let $(R,\m)$ be a $d$-dimensional regular local ring and 
let $(S,\n)$ be a $d$-dimensional regular local ring that birationally dominates $R$. 
Let $I$ be an $\m$-primary ideal of $R$ such that its transform $J = I^{S}$  in $S$  
is not equal to $S$, and let $V$ be a DVR that  birationally dominates $S$.
Then 
$$
V \in ~ \Rees_S J   ~ \iff   ~ V \in ~ \Rees_R I. 
$$
\end{proposition}

\begin{proof}
Since $R$ and $S$ are regular local rings, the hypothesis of Theorem~\ref{3.1} are satisfied. 
Hence $V$ is a Rees valuation ring of an ideal in $R$ or $S$ implies $V$ is a prime divisor
in the sense that the residue field of $V$ has transcendence degree $d-1$ over the
residue field of  $R$ or $S$.    
By \cite[Theorem~15.5]{M},  the
field $S/\n$ is algebraic over $R/\m$ .  Hence     $V$ is a prime divisor over $R$ 
if and only if $V$ is a prime divisor over $S$.

Let $I = (a_1, \ldots, a_n) R$.   Since $J = I^S$, there exists a nonzero $x \in S$    such that $J = (a_1 / x, \ldots, a_n / x) S$.
Since the ratios of the $a_i$ are the same as the ratios of the $a_i / x$,  Theorem~\ref{3.1}  
implies that   $V \in \Rees_S J$ if and only if $V \in \Rees_R I$.
\end{proof}

\begin{remark} With notation as in Theorem~\ref{3.1}, let $\overline{\frac{b_2}{b_1}}, \ldots,
\overline{\frac{b_d}{b_1}}$ denote the images of $\frac{b_2}{b_1}, \ldots, \frac{b_d}{b_1}$ in 
the residue field $k_v$ of $V$. An interesting integer associated with $V \in \Rees I$ and $b_1, \ldots b_d$ is the field degree 
$$
\Big [k_v: (R/\m)\Big (\overline{\big (\frac{b_2}{b_1}\big )},\ldots,  \overline{\big (\frac{b_d}{b_1}\big )}\Big )\Big ]
$$
\end{remark} 

\begin{lemma}\label{7.41}
Let  $(R, \m)$  be a regular local ring 
and let $L$ be an $\m$-primary ideal with $\ord_R(L) = 1$. 
Let $V$ denote the order valuation ring of $R$.  
We have 
$$ 
V ~\in ~\Rees_RL ~~\iff ~~ L ~= ~\m.
$$
\end{lemma}

\begin{proof}
By passing from $R$ to $R(u)$, where $u$ is an indeterminate over $R$, we may assume
that the residue field of $R$ is infinite, cf. \cite[page~159]{SH}. 
Let $J = (a_1, \ldots, a_d)R$ be a minimal reduction of $L$.  Since 
$J$ and $L$ have the same integral closure, we have $\ord_RJ = \ord_RL = 1$.
Thus at least one of the $a_i \notin \m^2$. If all the $a_i \notin \m^2$, then
$J = \m$ and hence $L = \m$. On the other hand,  if some $a_i \in \m^2$, then by Theorem~\ref{3.1}, 
$V \notin \Rees L$. 
\end{proof}

\begin{proposition}\label{7.5}    Assume notation  as in Setting~\ref{7.1} with  
$\dim R = \dim R_n  \ge 3$  and $R_0/\m_0 = R_n/\m_n$.  Let $I = P_{R_0R_n}$.
The following statements are equivalent:
\begin{enumerate}
\item
 There is no change of direction  from $R_0$ to $R_n$.
\item 
$\Rees_{R_0}  (I)   =      \Rees_{R_n} ( \m_n)$.
\item 
$ \ord_{R_0} I =1$.
\item
$\mathcal{B}(I)=\{1, 1, \ldots, 1,1\}.$
\item
There exist generators $x, y,  z  \ldots, w$ for $\m = \m_0$ such that
$$
\m_j ~  =  ~\big( x,  ~\frac{y}{x^{j}},   ~ \frac{z}{x^{j}},  ~   \ldots, ~ \frac{w}{x^{j}} \big) R_j
$$
for each $j$ with $1  \le j \le n$.
\item
There  exist generators $x, y,  z  \ldots, w$ for $\m = \m_0$ such that
$$
P_{R_0R_n}=\big(x^{n+1}, ~  y,  ~ z, ~    \ldots,  ~ w \big)R_0.
$$ 
\end{enumerate}
\end{proposition}

\begin{proof}
$(1)\Leftrightarrow (2)$: This follows from \cite[Theorem~6.8]{HK1}.\\
$(3)\Leftrightarrow (4)$: This follows from  Proposition~\ref{7.3}.\\
$(1)\Rightarrow (5)$: Since there is no change of direction in the local quadratic sequence 
from $R_0$ to $R_n$, by Remark~\ref{7.12},  we may choose an element $x \in \m_0$ such that $x$ is  
part of a minimal generating set for $\m_j$ for each 
$j$ with $1\leq j \leq n$.  

We prove item~5 holds by induction on $n$. Let $x, y', z', \ldots, w'$ be 
a regular system of parameters for $\m_0$.   
 Since $R_1$ is a localization of $R[\m/x]$ at a maximal ideal containing $x$ and 
$R/\m = R_1/\m_1$, there exist elements $a_1, b_1, \ldots, c_1 \in R$ 
such that 
$$ 
\m_1 ~=~ (x, ~ \frac{y'}{x}- a_1,   ~\frac{z'}{x} - b_1, ~\ldots, ~ \frac{w'}{x} - c_1)R_1.
$$
We take $y = y'-a_1x$, $z = z' - b_1x, \ldots, w = w' - c_1x$.  Then $\m = (x,y,z, \ldots, w)R$
and $\m_1 = (x, \frac{y}{x},  \frac{z}{x},  \ldots, \frac{w}{x})R_1$.  
This proves the case where $n=1$.  
Assume that item~5 holds for $n-1$. Then there exist 
elements $y', z', \ldots, w'$ such that  
$$
\m ~ =   ~ (x, y', z', \ldots, w')R   \qquad \text{ and } \qquad   \m_{n-1} ~  =  ~\big( x,  ~\frac{y'}{x^{n-1}},   ~ \frac{z'}{x^{n-1}},  ~  
 \ldots, ~ \frac{w'}{x^{n-1}} \big) R_{n-1}
$$
Since $R_n$ is a localization of $R_{n-1}[  \frac{\m_{n-1}}{x}]$ at a maximal ideal containing $x$ and
$R/m = R_n/\m_n$, there exist elements $a_n, b_n, \ldots, c_n \in R$ 
such that 
$$ 
\m_n ~=~ (x, ~ \frac{y'}{x^{n}}- a_n,   ~\frac{z'}{x^{n}} - b_n, ~\ldots, ~ \frac{w'}{x^{n}} - c_n)R_n.
$$
Then taking $y = y'-a_nx^n$, $z = z' - b_nx^n, \ldots, w = w' - c_nx^n$  completes an
inductive proof that item~1 implies item~5.

\noindent
$(5) \Leftrightarrow  (6)$:  This is a straightforward computation.\\
$(6) \Rightarrow (3)$: This is clear.\\
$(3)\Rightarrow (2)$:   
Let $V_j$ be the order valuation ring of $R_j$. Since $\ord_{R_j}(I^{R_j})=1$ 
for each  $j$ with $0 \leq j \leq n$, 
we have  $V_j \notin \Rees_{R_j}(I^{R_j})$, by Lemma~\ref{7.41}, 
and hence by \cite[Corollary~4.7]{HK1}, 
$\Rees_{R_{j+1}}(I^{R_{j+1}})=\Rees_{R_j}(I^{R_j})$. Thus we have 
$$
\Rees_{R_0}(I)=\Rees_{R_{1}}(I^{R_{1}})=\cdots=\Rees_{R_n}(I^{R_n})=\Rees_{R_n}(\m_n).
$$
\end{proof}

\begin{remark} \label{7.51} In the case where $R$ is a 2-dimensional regular local ring,
items 1, 3, 4, and 5 of Proposition~\ref{7.5} are equivalent and imply item~2. However, item~2
does not imply item~1.
\end{remark}

\section{$*$-simple complete monomial ideals } \label{c5}

In this  section we consider monomial ideals.  

\begin{definition}\label{5.1} 
Let $(R, \m)$ be an $d$-dimensional equicharacteristic regular local ring  and fix 
$d$ elements $x,y,\ldots, z$ such that 
$\m:=(x, y, \ldots, z)R$.   An ideal $I$ of $R$  
is said to be a {\bf monomial ideal}  if $I$ is generated by elements 
$x^ay^b \cdots z^c$ with $a,b,\ldots, c \in \N_0$.     Let 
$$
^xS ~= ~ R \big[ \frac{\m}{x} \big] = R \big[ \frac{y}{x}, ~\ldots, ~\frac{z}{x} \big] \qquad x_1 ~:= ~x,~\quad y_1:~= ~\frac{y}{x},~\quad \ldots,~\quad~z_1~:=\ ~\frac{z}{x}.
$$
If $I$ is a monomial ideal in $R$,   the transform of $I$ in $^xS$ is  generated by elements of the
form $x_1^ay_1^b \cdots z_1^c $ with $a,b,\ldots, c \in \N_0$.  
This motivates us to define an ideal $J$ of $^xS$ to be
a {\bf monomial ideal}  if $J$ is generated by monomials in 
$x_1, y_1, \ldots, z_1$.   We consider  
monomial quadratic transformations of $R$ defined as follows:   the ring 
$^xR=R\big[\frac{\m}{x}\big]_{(x,~ \frac{y}{x},~\ldots,~\frac{z}{x})}$ is a 
{\bf local monomial quadratic transformation}  of $R$ {\bf in the  $x$-direction}.
An ideal $J$ of $^xR$ is said to be a {\bf monomial ideal} if $J$ is generated by
monomials in $x_1, y_1, \ldots, z_1$. 

In a similar manner, we define   $^yR$, $\ldots$, $^zR$ to be the  
{\bf local  monomial quadratic transformations}  
of $R$ in the {\bf  $y$-direction}, $\ldots$, {\bf  $z$-direction}, respectively, if  
$$ 
^yR~=  ~R\big[\frac{\m}{y}\big]_{(\frac{x}{y},~y,~\ldots,~\frac{z}{y})}, \quad~
\ldots, \quad~
^zR=R\big[\frac{\m}{z}\big]_{(\frac{x}{z},~ \frac{y}{z},~\ldots,~z)}.
$$
We define an ideal of $^yR$, $\ldots$, $^zR$ to be a {\bf monomial ideal} if it is 
generated by monomials in the respective rings.  We refer to the elements in the  fixed set of 
minimal generators of the regular local ring as
{\bf variables}.

For a monomial ideal $I$ of one of these rings,  let  $\Delta (I)$ 
denote the set of monomial minimal generators of $I$.
\end{definition}

Notice that there are precisely $d$ distinct local monomial quadratic transformations  of $R$.    
If $I$ is a finitely supported complete monomial $\m$-primary 
ideal of $R$,  then the base points of $I$ in the first neighborhood of $R$  are 
a subset of $\{^xR,~ ^yR, ~\ldots, ~^zR \}$.    
Moreover, by repeating  the above process of monomial quadratic transformations,  we obtain 
more   information about the base points and point basis of  a finitely 
supported complete $\m$-primary monomial ideal of $R$.   There are, for example, 
at most $d^2$ base points of a monomial ideal
in the second neighborhood of  $R$.

\begin{setting}\label{Setting4}
Let $(R, \m)$ be a $d$-dimensional equicharacteristic regular local ring with $d \ge 3$ and fix a regular system of parameters $x, y, \ldots, z$ for $R$.
Let $R_1 := {^xR}$ be the local monomial quadratic transform of $R$ in the $x$-direction, where $\m_1 := (x_1, y_1, \ldots, z_1) R_1$ is the maximal ideal of $R_1$ as in Definition~\ref{5.1}.
\end{setting}
  
 We observe in Remark~\ref{monoprop},  that many   of the properties of monomial ideals  of a localized polynomial ring    
 over a field also hold for the the  monomial ideals of Setting~\ref{Setting4}.

\begin{remark}   \label{monoprop}  Let $I$ and $J$  be  monomial $\m$-primary  ideals  of the ring $R$ 
of Setting~\ref{Setting4}. 
\begin{enumerate}
\item  A monomial  of $R$  is in   $I$ if and only if it is a multiple of 
a monomial in $\Delta(I)$.
\item  Let $K$ denote any one of the ideals $I + J$, $I J$, $(I : J)$, and $I \cap J$.
Then $K$ is also a monomial ideal.
\item  If $I$ is complete, then $(I:J)$ is complete.
\item\label{monoprop4} The integral closure $\overline I$ of $I$ is again a monomial ideal.
\item\label{monoprop5}
If a power of one of $x, y, \ldots,$ or $z$ is in the integral closure of $I$ then it is also in $I$.
\end{enumerate}
\end{remark}

\begin{proof}  Item~1 and item~2 are  Lemmas~6, 7  and  Theorem~6  of  \cite{T},   and item~3 is Remark~1.3.2 of \cite{SH}.
The proof of item~4 is given in \cite{KS}.
% since $R$ is an equicharacteristic regular local ring,    it follows from  \cite[Theorem~15]{Co} that 
% the completion $\widehat{R}$  of $R$ is a formal power series ring  $k [[x, y, \ldots, z]]$ over a field $k$. 
%Consider the localized polynomial subring $R' = k[x, y, \ldots, z]_{(x, y, \ldots, z)}$ of $\widehat{R}$.  
%The completion  $\widehat{R'}$ of $R'$  is $\widehat{R}$,  and  the ideal $I' = I \widehat{R} \cap R'$ is a monomial ideal
%of $R'$  with the property that  $\Delta (I') = \Delta (I)$.  The integral closure $\overline{I'}$ of $I'$  is a monomial ideal of $R'$  
%by \cite[Proposition~1.4.2]{SH}.  Thus  $\overline{I'} \widehat{R}$ is a monomial idea of $\widehat R$.   Since $R'$ is excellent, 
% $\overline{I'} \widehat{R}$ is a complete ideal  by \cite[page~800]{L2},  and we have $$ \overline{I'} \widehat{R} ~ =
%~\overline{I' \widehat{R}} ~ =  ~\overline{I \widehat{R}}. $$ Since $\widehat{R}$ is faithfully flat over $R$,  \cite[Proposition 1.6.2]{SH}  implies that 
 %$\overline{I \widehat{R}} \cap R = \overline{I}$, where $\overline I$ is the integral closure of $I$.
%Since $\overline{I'} \widehat{R}$ is generated by monomials in $R$,  it follows that  $\overline{I}$ is a monomial ideal.

 For the proof of item~5, we use 
that the quotient ring obtained by 
 going modulo the ideal generated by the other $d-1$ variables is a PID. 
Assume for example that  $x^n \in \overline I$. Then $I + (y,\ldots, z)R$ is an integrally 
closed monomial ideal since its image $ \frac{I + (y,\ldots,z)R}{(y,\ldots,z)R}$ is an
integrally closed ideal. Hence $\overline I \subseteq I + (y,\ldots,z)R$.  Moreover, $x^n   \in  I + (y,\ldots,z)R$ 
and $I + (y,\ldots,z)R$ a monomial ideal implies $x^n \in I$.
\end{proof}

\begin{proposition}\label{contract7}
Assume the notation of Setting~\ref{Setting4}, 
and let $I_1$ be a monomial ideal of $R_1$ that contains a power of $x = x_1$.
Then $I_1 \cap R$ is a monomial ideal of $R$.
\end{proposition}

\begin{proof} 
Since $R$ is equicharacteristic, the completion $\widehat R$ of $R$   has a coefficient field $k$ \cite[(31.1)]{N}. 
Since $R$ is regular, $\widehat R$ is regular and is the $d$-dimensional formal power series ring $k[[x, y, \ldots, z]]$.
Since $R/\m = R_1/\m_1$, the completion $\widehat {R_1}$ of $R_1$ is the formal power series ring
$k[[x, y_1, \ldots, z_1]]$,  and the local   inclusion map  $R  \hookrightarrow R_1$    extends to a
homomorphism on completions $\phi: \widehat{R} \rightarrow \widehat{R_1}$.
Commutativity of the diagram
$$\begin{tikzpicture}
\node[name=RHat] at (-1, 1) {$\widehat{R}$};
\node[name=k1] at (-1, -1) {$R / \m$};
\node[name=R1Hat] at (1, 1) {$\widehat{R'}$};
\node[name=k2] at (1, -1) {$R_1 / \m_1$};

\draw[->]
	(RHat) edge (k1)
	(R1Hat) edge (k2)
	(RHat) edge node[above] {$\phi$} (R1Hat)
	(k1) edge node[above] {$\simeq$} (k2);
\end{tikzpicture}$$
implies that $\phi$ is a $k$-algebra homomorphism.

$$ 
\phi:  \widehat R \longrightarrow  \widehat{R_1} \quad \text{ where } \quad x ~\mapsto ~ x ,\quad y ~ \mapsto  ~xy_1, \quad \ldots \quad , z ~\mapsto ~ xz_1.
$$
Uniqueness of expression for $f \in \widehat{R}$ as a power series in $k [[x, y, \ldots, z]]$ implies that $\phi$ is injective.
The subring   $R' := k [x, y, \ldots, z]_{(x, y, \ldots, z)}$  of $\widehat R$ is  a localized polynomial ring in $d$-variables over $k$, 
and we have $\widehat{R'} = k[[x,y, \ldots, z]] = \widehat R$.  
Similarly,  the subring  $R'_1 : = k [x, y_1, \ldots, z_1]_{(x, y_1, \ldots, z_1)}$ of $\widehat{R_1}$ is a 
localized polynomial ring in $d$ variables over the field $k$  , and we have  $\widehat {R_1} = \widehat {R'_1}$.

Note that the set $\Delta (I_1)$ of minimal monomial generators of $I_1$   in $R_1$ is contained in  the rings $\widehat{R_1}$ and $R'_1$.
Consider the diagram,

$$\begin{tikzpicture}
\node[name=RHat] at (0, 0) {$\widehat{R}$};
\node[name=R] at (-1.5, -1) {$R$};
\node[name=RPrime] at (1.5, -1) {$R'$};

\node[name=R1Hat] at (0, 2) {$\widehat{R_1}$};
\node[name=R1] at (-1.5, 1) {$R_1$};
\node[name=R1Prime] at (1.5, 1) {$R'_1$};

\draw[->]
	(R) edge (RHat)
	(R) edge (R1)
	(R1) edge (R1Hat)

	(RPrime) edge (RHat)
	(RPrime) edge (R1Prime)
	(R1Prime) edge (R1Hat)

	(RHat) edge node[right] {$\phi$} (R1Hat);
\end{tikzpicture}$$

Since $\widehat{R_1}$ is faithfully flat over $R_1$,   we have 
$$
I_1 ~=  ~\Delta(I_1)R_1  ~= ~ \Delta(I_1)\widehat{R_1} \cap R_1.
$$
Since  $\Delta (I_1) \subset R'_1$ and  $\widehat{R_1}$ is  faithfully flat over $R'_1$, we have 
$I_1 \widehat{R_1} \cap R'_1 = \Delta (I_1) R'_1$.

Define a $\Z^d$-grading on the polynomial ring $A = k[x,y, \ldots, z]$ and its localization $A[\frac{1}{x}]$ by 
giving $x$ weight $(1,0, \ldots, 0)$, $y$ weight $(0,1,0, \ldots, 0)$, $\ldots$, and $z$ weight $(0,0,\ldots, 1)$.
The polynomial ring $A_1 = k[x, \frac{y}{x}, \ldots, \frac{z}{x}]$ is a graded subring of $A[\frac{1}{x}]$,    
and $\Delta(I_1)$ is  a subset of $A_1$. 
Since $A$ is a graded subring of $A_1$ and the graded ideals of $A$ with respect to this multi-grading are precisely the monomial ideals,
we have $J : = \Delta(I_1)A_1 \cap A$ is a monomial ideal of $A$.   Since $\Delta(I_1)$ contains a power of $x$, the ideal $J$ is 
primary for the maximal ideal $(x,y, \ldots, z)A$.  Hence $JR'   = \Delta(J)R'$ is  a 
monomial ideal of $R'$ that is  primary for the
maximal ideal $\m'$ of $R'$, and we have $\Delta(J)R' = \Delta(I_1)R_1' \cap R'$.

Since the $\m'$-primary ideals of $R'$ are in one-to-one inclusion preserving 
correspondence with the $\widehat{\m}$-primary ideals of $\widehat R$,  we have 
$I_1 \widehat{R_1} \cap \widehat{R} = \Delta(J) \widehat{R}$.
It follows that $\Delta(J )\widehat{R} \cap R = I_1 \cap R$.
Since $\widehat{R}$ is faithfully flat over $R$ and $\Delta(J) \subset R$,  it follows that $\Delta(J)R =  I_1 \cap R$ 
is a monomial ideal of $R$.
\end{proof}

Let $R_1$ denote the ring $^xR$ of Definition~\ref{5.1}, 
let $I_1$ be an $\m_1$-primary monomial ideal in $R_1$,
and let $\nu_x$ denote the $x$-adic valuation of $R$ on its field of fractions $\mathcal Q(R)$.
Thus $\nu_x (x) = 1$, $\nu_x (y_1) = \cdots = \nu_x(z_1) =  -1$.
   For each  monomial 
$\alpha ~:=  ~x_1^{\alpha_x}y_1^{\alpha_y} \cdots z_1^{\alpha_z} ~ \in  ~ \Delta(I_1)$,   we have  
$$
 \nu_x (\alpha) ~ =  \alpha_x~-~(\alpha_y~+~\cdots~+~\alpha_z).
$$
Define   the integer $\delta(I_1)$ as follows:
$$
 \delta(I_1)~~:= ~~ \max~\Big\{ - \nu_x (\alpha)~\vert~ \alpha ~\in ~ \Delta(I_1) \Big\}.
$$
Thus $x^{\delta (I_1)} \alpha \in R$ for each $\alpha \in \Delta (I_1)$, and $\delta (I_1)$ is the smallest integer with this property.

In analogy with work of  Gately   \cite[page~2844]{G2}      in the case where $R$ is  a localized polynomial ring  in three variables 
 over a field,   we define   the complete inverse transform $\CIT(I_1)$  of $I_1$ to be  
 the integral closure of the ideal  $J$, where   
\begin{equation} \label{5.11}
J  ~:=   ~\Big( \big\{ x^{\delta (I_1)} \alpha = x^{\delta(I_1) + \nu_x (\alpha)}y^{\alpha_y} \cdots z^{\alpha_z}~\vert~
\alpha \in \Delta(I_1)  \big\},
\quad y^{\delta(I_1)},
\quad \ldots,
\quad  z^{\delta(I_1)} \Big )R.
\end{equation}

We observe in Lemma~\ref{5.14}  that $\CIT (I_1)$ has the following properties:  

\begin{lemma}\label{5.14}
Assume the notation of Setting~\ref{Setting4}, and let 
$I_1$ be an  $\m_1$-primary  complete  monomial ideal. There exist  integers 
$n_x, n_y, \ldots,  n_z$ such that $x_1^{n_x}, y_1^{n_y}, \ldots, z_1^{n_z} \in \Delta (I_1)$.
Let $I := \CIT (I_1)$.   Then:
\begin{enumerate}
\item
$\delta (I_1) = \max (n_y, \ldots, n_z)$.
\item
$\ord_{R} (I) = \delta (I_1)$.
\item
$x^{n_x + \delta (I_1)} \in \Delta (I)$.
\end{enumerate}
\end{lemma}

\begin{proof}~
    To prove item~1, let   $r := \max (n_y, \ldots, n_z)$.   By definition, we have  $\delta (I_1) \ge r$.
Since $I_1$ is integrally closed,  we have 
$$
I_1  ~\supset ~ \overline{(y_1^{n_y}, ~ \ldots, ~ z_1^{n_z})} ~ \supset  ~\overline{(y_1^{r}, ~ \ldots, ~ z_1^{r})}~ =  ~(y_1, ~ \ldots, ~ z_1)^{r}.
$$
The last equality follows because $(y_1, \ldots, z_1)$ is a normal ideal of $R_1$.
Thus whenever $\alpha_y + \ldots + \alpha_z = r$,  we have $y_1^{\alpha_y} \cdots z_1^{\alpha_z} \in I_1$.
Hence for  every element $x_1^{\alpha_x} y_1^{\alpha_y} \cdots z_1^{\alpha_z} \in \Delta (I_1)$, 
we  have $\alpha_y + \ldots + \alpha_z \le r$, so in particular, $\alpha_y + \ldots + \alpha_z - \alpha_x \le r$, and item~1 holds. 

To prove item~2,  observe that by construction of $\CIT (I_1)$, we have $I = \overline {J}$, 
where $J$ is as defined in Equation~\ref{5.11}. 
We have $y^{\delta(I_1)}  \in J$, and  
$\ord_{R}(y^{\delta(I_1)})=\delta(I_1)$.    Also we have
$$
\ord_{R}(x^{\delta(I_1)+ \nu_x (\alpha) }y^{\alpha_y} \cdots z^{\alpha_z})
=\delta(I_1)+\alpha_x \geq \delta(I_1).
$$
Hence $\ord_{R}(J)=\delta(I_1)$. Since 
 $\ord_{R}(\overline{J})=\ord_{R}(J)$, we have $ \ord_{R}(I)= \delta(I_1)$.

Since $x_1^{n_x} \in \Delta (I_1)$  the definition of $J$ gives $x^{\delta (I_1) + n_x} \in J$.   
 Remark~\ref{monoprop} implies that    $x^n \in \Delta(J)$ if and only 
if $x^n \in \Delta(\overline J)$. 
Since every other monomial in Equation~\ref{5.11} is divisible by one of the variables  $y, \ldots, z$, it follows that 
$x^{\delta (I_1) + n_x} \in \Delta (J)$.   This proves  item~3.
\end{proof} 

\begin{remark}\label{compat2}
Assume the notation of Setting~\ref{Setting4}, and let 
$I_1$ be a complete $\m_1$-primary monomial ideal in $R_1$.    Let 
the ideal  $J$ be as in   Equation~\ref{5.11}.    To see that 
$I = \overline J$    is   the inverse transform of $I_1$ as 
defined by Lipman in Lemma 2.3 of \cite{L},   it suffices to observe that with 
the notation of Definition~\ref{5.1},  we have 
\begin{enumerate} 
 \item  The transform of $I$ in $^xS$ is a monomial ideal that localizes in $R_1$ to the ideal $I_1$.
This is clear by definition of $I$.
\item The transform of $I$ in any of the other $d-1$ affine components $^yS = R [\frac{\m}{y}]$, $\ldots$, and $^zS = R [\frac{\m}{z}]$ 
is  the unit ideal.
This is clear because $y^{\delta (I_1)}$, $\ldots$, and $z^{\delta (I_1)}$ are in $I$.
\item  The ideal $I$ is not a $*$-multiple of $\m$. 
\end{enumerate}
To see that $I$ is not a $*$-multiple of $\m$, let $L = I : \m$, and assume by way of contradiction 
that $I$ is a $*$-multiple of $\m$, say there a complete ideal $K$ such that $I = \m * K$.
Since $\m K \subset I$,   we have $K \subset L$.
Thus $\m * L \subset \m * K  =I   \subset  \m * L$, so we may assume $\m * L = I$.
The ideal $L$ is a complete monomial ideal, and the ideal $\m L$ is a monomial ideal.
Since $y^{\delta (I_1)}, \ldots, z^{\delta (I_1)} \in \m * L$, it follows that $y^{\delta (I_1)}, \ldots, z^{\delta (I_1)} \in \m L$ by Remark~\ref{monoprop}.\ref{monoprop5}.
But this implies that $y^{\delta (I_1) - 1}, \ldots, z^{\delta (I_1) - 1} \in L$.
Since the transform $L^{R_1}$ of $L$ in $R_1$ is $I_1$, it follows that $y_1^{\delta (I_1) - 1}, \ldots, z_1^{\delta (I_1) - 1} \in I_1$, 
which is a contradiction.    
\end{remark}

\begin{remark}\label{compat3}
Assume the notation of Setting~\ref{Setting4}, and let $I$ be an $\m$-primary complete monomial ideal with exactly one base point $R_1$ in its first neighborhood.
 Lemma 2.3 of \cite{L} implies that there exists a nonnegative integer $n$ such that,
	$$I = \m^n * \CIT (I^{R_1}).$$
Thus if $\ord_{R} (I) = r$, then Equation~\ref{5.11} implies that $(y, \ldots, z)^{r} \subset I$.
\end{remark}

\begin{proposition}\label{eq9}
Assume the notation of Setting~\ref{Setting4}, and let $I_1$ be a complete 
$\m_1$-primary monomial ideal of $R_1$. Let $I  = \CIT (I_1)$ and $\delta = \delta(I_1)$. 
Then:
\begin{equation}\label{5.129}
I~ = ~
\Big(\{  \alpha = x^{\alpha_x} y^{\alpha_y} \cdots z^{\alpha_z} ~|~ x^{- \delta} \alpha = x^{\alpha_x + \alpha_y + \ldots + \alpha_z - \delta} y_1^{\alpha_y} \cdots z_1^{\alpha_z} \in I_1\} \Big)R.
\end{equation}
Thus for $a, b, \ldots, c \in \N_0$  with   $ a +b + \ldots + c  = \delta $,  we have 
$$
x^{a} y^{b} \cdots z^{c}  ~ \in  ~ I      ~\iff    ~    y_1^{b} \cdots z_1^{c}  ~ \in  ~   I_1.
$$
Hence the monomials in $I$ of minimal order are determined by the monomials 
in $I_1$ involving only the $d - 1$ elements $y_1, \ldots, z_1$.
\end{proposition}

\begin{proof}    
By Remark~\ref{monoprop}.\ref{monoprop4}, $I$ is a monomial ideal, and by Remark~\ref{compat2},   
$I$ is equal to $x^{\delta} I_1 \cap R$.
Let $x^a y^b \cdots z^c \in R$ be a monomial.   We have  $x^a y^b \dots z^c \in I$ if and only 
if $x^a y^b \cdots z^c \in x^{\delta} I_1$.
Rewriting $x^a y^b \cdots z^c = x^{a + b + \ldots + c} y_1^{b} \cdots z_1^{c}$, it follows   that 
$x^a y^b \cdots z^c \in x^{\delta} I_1  \iff    x^{a - \delta  + b + \ldots + c} y_1^{b} \cdots z_1^{c} \in I_1$.  
The final assertion  is an immediate consequence of Equation~\ref{5.129}.
\end{proof}

\begin{lemma}\label{mingens}
Assume the notation of Setting~\ref{Setting4}, and let $I_1$ be a complete $\m_1$-primary 
monomial ideal of $R_1$.
Let $I := \CIT (I_1)$ in $R$  and let $\delta = \delta(I_1)$.

\begin{enumerate}
\item
For every $\alpha = x^{a} y_1^{b} \cdots z_1^{c} \in \Delta (I_1)$, we have
$x^{\delta } \alpha = x^{\delta  + a - b - \ldots - c} y^{b} \cdots z^{c} \in \Delta (I)$.
Thus the map 
\begin{equation} \label{5.131}
\phi : \Delta (I_1) \longrightarrow  \Delta (I) \qquad \text{  defined by } \qquad 
\phi(\alpha)  = x^{\delta } \alpha
\end{equation}
is a one-to-one map from $\Delta (I_1)$ 
into $\Delta (I)$.
In particular, $\mu (I) \ge \mu (I_1)$.

\item
Every monomial in $\Delta (I)$ has the form $y_1^{e} \cdots z_1^{f} x^{\delta } \gamma$ for some 
$\gamma = x^{a} y_1^{b} \cdots z_1^{c} \in \Delta (I_1)$, where $e + \ldots + f \le \delta  + a - (b + \ldots + c)$.
Thus every minimal monomial generator  of $I$ is obtained from the set $x^{\delta } \Delta (I_1)$ 
by possibly replacing  $x^i$ by  $y^j \cdots z^k$, where $i = j+ \cdots + k \le (j + \ldots + k) + (b + \ldots + c) \le \delta $.  
\end{enumerate}
\end{lemma}

\begin{proof}
Recall by Equation~\ref{5.129} that the monomials in $I$ are the monomials in $x^{\delta} I_1$ 
that are in $R$.

To see item 1, let $\alpha = x^{\alpha_x} y_1^{\alpha_y} \cdots z_1^{\alpha_z} \in \Delta (I_1)$ 
be as in the statement  of item~1.
Equation~\ref{5.11} implies that  
$x^{\delta} \alpha = x^{\alpha_x - (\alpha_y + \ldots + \alpha_z) + \delta}y^{\alpha_y} \cdots z^{\alpha_z} \in I$. 
We  show that $x^{\delta} \alpha$  is in $\Delta(I)$.
Let $\beta = x^{\beta_x} y^{\beta_y} \cdots z^{\beta_z} \in I$ be a monomial 
that  divides $x^\delta \alpha$.  Then 
$$
\beta_x \le \alpha_x - (\alpha_y + \ldots + \alpha_z) + \delta,\quad 
\beta_y \le \alpha_y,
\quad \ldots,
\quad \beta_z \le \alpha_z.
$$
It follows that  
$x^{- \delta} \beta = x^{\beta_x + (\beta_y + \ldots + \beta_z) - \delta} y_1^{\beta_y} \cdots z_1^{\beta_z}$ 
is in $I_1$, and we have  
$$
\beta_x + (\beta_y + \ldots + \beta_z) - \delta \le 
(\delta + \alpha_x - (\alpha_y + \ldots + \alpha_z)) + (\alpha_y + \ldots + \alpha_z) - \delta = \alpha_x.
$$
Hence  $x^{- \delta} \beta$ divides  $\alpha$ in $R_1$.  Since 
$\alpha \in \Delta(I_1)$, we have  $x^{- \delta} \beta = \alpha$ and 
$\beta = x^{\delta} \alpha$. This proves item~1.

To see item 2, let $\alpha = x^{\alpha_x} y^{\alpha_y} \cdots z^{\alpha_z} \in \Delta (I)$, and consider  its transform 
 $x^{- \delta} \alpha = x^{\alpha_x + (\alpha_y + \ldots + \alpha_z) - \delta} y_1^{\alpha_y} \cdots z_1^{\alpha_z} \in I_1$.
Then $x^{- \delta} \alpha$ is divisible by some $\beta = x^{\beta_x} y_1^{\beta_y} \cdots z_1^{\beta_z} \in \Delta (I_1)$  
and  we have the inequalities,
	$$\beta_x \le \alpha_x + (\alpha_y + \ldots + \alpha_z) - \delta, \quad \beta_y \le \alpha_y, \quad \ldots, \quad \beta_z \le \alpha_z.$$
Consider the integer $s := \beta_x - (\alpha_y + \ldots + \alpha_z) + \delta$.  Then  $s \le \alpha_x$.
Let $\gamma = y_1^{\alpha_y - \beta_y} \cdots z_1^{\alpha_z - \beta_z}$.
Then  $\gamma \beta = x^{\beta_x} y_1^{\alpha_y} \cdots z_1^{\alpha_z} \in I_1$ and $x^{\delta} x^{- s} \gamma \beta = y^{\alpha_y} \cdots z^{\alpha_z} \in R$.
We first show that $s \ge 0$.

Suppose by way of contradiction that $s < 0$.
Then   $x^{-s} \gamma \beta \in I_1$  and  it follows   that 
$x^{\delta} x^{-s} \gamma \beta = y^{\alpha_y} \cdots z^{\alpha_z} \in I$.   
Since $\alpha \in \Delta(I)$, this monomial is $\alpha$, and we have $\alpha_x = 0$.
However, $(y, \ldots, z)^{\delta } \subset I$  implies  $\alpha_y + \ldots + \alpha_z \le \delta$, 
and this implies  $s \ge 0$,  a contradiction.

Thus $s \ge 0$.
That $\gamma \beta \in I_1$ implies $x^{\delta} \gamma \beta = x^{s} y^{\alpha_y} \cdots z^{\alpha_z} \in I$.  
Since $\alpha \in \Delta(I)$, it follows that $\alpha$ divides this monomial, 
so we have $\alpha_x \le s$, and by construction, $s \le \alpha_x$, so $s = \alpha_x$.
This proves item~2.
\end{proof}

 For a class of monomial ideals $I$ that includes 
complete inverse transforms,  we prove in Theorem~\ref{mingens3}  
that  the minimal number of generators of $I$ is completely determined by the order of $I$.

\begin{theorem}\label{mingens3}
Let  $(R,\m)$ be a $d$-dimensional  equicharacteristic regular local ring,  and 
fix $d$ elements $x,y, \ldots, z$ such
that $\m = (x, y, \ldots, z)R$.    Let   $I$ be an $\m$-primary monomial ideal   
with  $\ord_{R} (I) = r$.   If 
$(y, \ldots, z)^{r} \subset I$  and $I$ 
 is contracted from $S = R[\frac{\m}{x}]$,   
 then $\mu (I) = \mu (\m^r) = \binom{d+r-1}{r}$.
\end{theorem}

\begin{proof}
Let $\mathcal{S}$ denote the set of monomials in $y, \ldots, z$ of degree less than or equal to $r$.    To prove the theorem, it
suffices to  show that the elements of $\Delta (I)$ are in one-to-one correspondence 
with the  elements of $\mathcal S$.
Since $I$ is $\m$-primary, for each  monomial $\alpha \in \mathcal{S}$ there is 
a nonnegative integer $c$ such that $x^c \alpha \in I$.
By choosing  $c$ to be minimal with this property,   we  obtain a one-to-one map of sets $\varphi : \mathcal{S} \rightarrow I$.
Notice that for each  monomial $\alpha \in \mathcal{S}$ of degree $r$,  we have  $\varphi (\alpha) = \alpha$ and $\alpha \in \Delta (I)$.

Given a monomial $\beta = x^{\beta_x} y^{\beta_y} \cdots z^{\beta_z} \in I$, set $\alpha = y^{\beta_y} \cdots z^{\beta_z}$.
If the degree of $\alpha$ is greater than $r$, then $\alpha$ is divisible 
by an element in $\varphi (\mathcal S)$.
If the degree of $\alpha$ is less than or equal to $r$, then $\beta$ is 
divisible by  $\varphi (\alpha)$.
We conclude that     $I    \subseteq   \varphi(\mathcal S) R$.  Since the elements 
in $\varphi(\mathcal S)$ are
monomials,  it follows that $\Delta (I)  \subseteq     \varphi(\mathcal S)$.

It remains  to  show that every element in $\varphi (\mathcal{S})$ is in $\Delta (I)$.
Suppose by way of contradiction that there  exists 
$\alpha = y^{\alpha_y} \cdots z^{\alpha_z} \in \mathcal{S}$ 
such that $\varphi (\alpha) = x^{c} \alpha   \in I$ is not in $\Delta (I)$.
There exists a  monomial $\beta = x^{\beta_x} y^{\beta_y} \cdots z^{\beta_z} \in I$ that  properly  divides $x^c\alpha$.
Take $\beta$ so that $\beta_x$ is minimal among monomials in $I$ that   properly  divide $x^c\alpha$.
The minimality of $c$ implies that for some variable $w$ other than $x$, $\beta_w < \alpha_w$.
We may assume without loss of generality that $w = y$.
If $\beta_x = 0$, then $r = \beta_y + \ldots + \beta_z < \alpha_y + \ldots  +\alpha_z$, 
a contradiction to the assumption that $\alpha$ has degree at most $r$.
The fact that $I$ is contracted from $S$ implies that 
$x^{\beta_x - 1} y y^{\beta_y} \cdots z^{\beta_z}$ is an element in $I$ that 
properly divides $x^{c} \alpha$.
This contradicts the minimality of $\beta_x$ and thus completes the proof of Theorem~\ref{mingens3}.
\end{proof}

\begin{corollary}\label{mingens4}
Assume the notation of Setting~\ref{Setting4}, and let $I$ be a complete $\m$-primary monomial ideal 
of order $r$ in $R$ whose only base point in the first neighborhood of $R$ is $R_1$.
Then $\mu (I) = \mu (\m^r)$.
\end{corollary}
\begin{proof}
In view of Remark~\ref{compat3}, this follows from Theorem~\ref{mingens3}.
\end{proof}

\begin{example} \label{examtable} 
Assume the notation of Setting~\ref{Setting4} with $d = 3$.
Consider the ideal
$$
J_1  ~= ~  (x^2,~  xy_1,~ xz_1, ~ y_1z_1,~ y_1^3,~ z_1^3)R_1.
$$
Then 
$$
I~=~ \CIT(J_1) ~=~(x^5, ~x^3y,~x^3z,~x^2y^2,~xyz,~x^2z^2, ~(y,~z)^3)R.
$$
%For this ideal $I$, the following table lists the various values of the integer $c$
%given in the proof of   Theorem~\ref{mingens3}.  

In the following tables, the entry in the $i$-th column and $j$-th row gives the integer $c$ such that $x^c y^i z^j \in \Delta (I)$.
The table on the left shows the image of $\Delta (J_1)$ under the map $\phi$ defined in Lemma~\ref{mingens}, and the table on the right shows all of the elements of $\Delta (I)$, obtained from the left table by converting powers of $x$ to powers of $y, z$ as in Lemma~\ref{mingens}.

\begin{center}
\begin{tabular}{c | | c | c | c | c}
& 0 & 1 & 2 & 3 \\
\hline \hline 0 & 5 & 3 & & 0 \\
\hline 1 & 3 & 1 & & \\
\hline 2 & & & & \\
\hline 3 & 0 & & & \\
\end{tabular}
\hspace{2cm}
\begin{tabular}{c | | c | c | c | c}
& 0 & 1 & 2 & 3 \\
\hline \hline 0 & 5 & 3 & 2 & 0 \\
\hline 1 & 3 & 1 & 0 & \\
\hline 2 & 2 & 0 & & \\
\hline 3 & 0 & & & \\
\end{tabular}
\end{center}

\end{example}

\begin{theorem}  \label{shapeofideals}
Assume the notation of Setting~\ref{Setting4}.
Let $I_1$ be a
complete $\m_1$-primary monomial ideal, let $I := \CIT (I_1)$ in $R$,
and let $\delta = \delta(I_1)$.  
Consider the following statements:
\begin{enumerate}
\item $\mu (I) = \mu (I_1)$.
\item $\ord_{R} (I) = \ord_{R_1} (I_1)$.
\item For every $y_1^{b} \cdots z_1^{c} \in \Delta (I_1)$, $b + \ldots + c = \delta$.
\item $\ord_{R}$ is not a Rees valuation of $I$.
\end{enumerate}
Then $(1) \iff (2),   ~~ (2)  \implies (3)$ and $(3) \iff (4)$.
\end{theorem}

\begin{proof}
Lemma~\ref{mingens} implies that item~1 is equivalent to  the map 
$\phi: \Delta(I_1) \to \Delta(I)$ of Equation~\ref{5.131} is
surjective.   
Thus    $\mu(I) = \mu(I_1)$  if  and only if 
\begin{equation}\label{mingens2}
\Delta (I) ~=  ~ \{~ x^{\delta  + a - (b + \ldots + c)} y^{b} \cdots z^c  ~|~ x^{a} y_1^{b} \cdots z_1^{c} \in \Delta (I_1) ~ \}.
\end{equation}

To see that item~1 implies item~2, assume item~1, and suppose by way of contradiction 
that item~2 does not hold.
Then there exists an element 
$\alpha = x^{\alpha_x} y_1^{\alpha_y} \cdots z_1^{\alpha_z} \in \Delta (I_1)$ 
such that $\alpha_x + (\alpha_y + \ldots + \alpha_z) = \ord_{R_1} (\alpha) < \delta$.
Among all such elements, take $\alpha$ to have minimal $\alpha_x$.
The element 
$x^{\delta} \alpha = x^{\alpha_x + \delta - (\alpha_y + \ldots + \alpha_z)} y^{\alpha_y} \cdots
z^{\alpha_z}$ is in $\Delta (I)$ by Equation \ref{mingens2}.
Our assumption implies that 
$\delta + \alpha_x - (\alpha_y + \ldots + \alpha_z) > 0$, so $y_1 x^{\delta} \alpha \in I$, 
and hence is divisible by an element in $\Delta (I)$.
Since $\mu (I) = \mu (I_1)$, $y_1 x^{\delta} \alpha$ is divisible 
by some $x^{\delta} \beta$, where 
$\beta = x^{\beta_x} y_1^{\beta_y} \cdots z_1^{\beta_z} \in \Delta (I_1)$.
That is,
$$
\beta_x + \delta - (\beta_y + \ldots + \beta_z) ~\le ~ \alpha_x + \delta - (1 + \alpha_y + \ldots + \alpha_z), 
$$
\begin{equation}\label{soi1}
\beta_y ~\le ~ \alpha_y + 1,
\quad \ldots,
\quad \beta_z ~ \le ~ \alpha_z.
\end{equation}
In Equation~\ref{soi1}, the  lower dots  represent the fact that for every 
one of our fixed set of minimal generators $w$ for $\m$ other than $x$ or 
$y$, $\beta_w \le \alpha_w$.
Hence
$$
\beta_x ~ \le ~ \alpha_x + (\beta_y + \ldots + \beta_z) - (1 + \alpha_y + \ldots + \alpha_z) ~ \le ~ \alpha_x.
$$
We have either $\beta_x = \alpha_x$ or $\beta_x < \alpha_x$.

Suppose that $\beta_x = \alpha_x$.
Then each of the inequalities in Equation~\ref{soi1} is an equality.
Thus $\alpha$ properly divides $\beta$ in $I_1$, a contradiction to the 
fact that $\beta \in \Delta (I_1)$.

Thus we must have $\beta_x < \alpha_x$, that is, $\beta_x \le \alpha_x - 1$.
Then 	
$$
\ord_{R_1}(\beta) ~ = ~ \beta_x + (\beta_y + \ldots + \beta_z) ~\le ~ 
(\alpha_x - 1) + (1 + \alpha_y + \ldots + \alpha_z) ~\le~
\alpha_x + (\alpha_y + \ldots + \alpha_z) ~ < ~ \delta.
$$
This contradicts the choice of $\alpha$ and completes the proof that item~1 implies item~2.

Assume item~2 holds. To prove item~3 also holds,  let 
$y_1^{\alpha_y} \cdots z_1^{\alpha_z} \in \Delta (I_1)$.  
Lemma~\ref{5.14} implies that $\delta   = \ord_{R} (I)$.  Hence we have
$$
\ord_{R_1} (I_1)~ \le  ~\alpha_y + \ldots + \alpha_z  ~\le  ~\delta  ~ = 
~\ord_{R} (I)~ =  ~\ord_{R} (I_1),
$$
 so equality holds throughout, and item~3 holds.

We next show that item~2 implies item~1.  Let $S_2 := R_1 [\frac{\m_1}{x}]$.  
We  show that $I_1 S_2 \cap R_1 = I_1$.  We have 
$y_1^\delta, \ldots, z_1^\delta \in \Delta(I_1)$ and   $x^{n_x} \in \Delta( I_1)$ for some 
integer $n_x \ge \delta$.
We consider the extension of $I_1$ in the blowup  
$\Proj(R_1[\m_1t])$ of $\m_1$.
Since $\ord_{R_1} (I_1) = \delta$ and $y_1^{\delta}, \ldots, z_1^\delta  \in I_1$, 
the extension of $I_1$ in every monomial quadratic transform of $R_1$ 
other than in the $x_1$-direction is the same as the extension of $\m_1^{\delta}$.
Thus $(S_2)_{(x, \frac{y_1}{x}, \ldots, \frac{z_1}{x})}$ is the unique base 
point of $I_1$ in the first neighborhood of $R_1$.
Corollary~\ref{mingens4} implies that $\mu (I_1) = \mu (\m_1^\delta)$ 
and $\mu (I) = \mu (\m^\delta)$, thus proving item~1.

To see that item 3 implies item 4, consider the monomials of $I$ of minimal order in $R$.
Proposition~\ref{eq9} implies each  such monomial has the form 
$y^{\alpha_y} \cdots z^{\alpha_z}$, where $\alpha_y + \ldots + \alpha_z = \delta$.

Let $w_1, \ldots, w_{d-1}$ denote the variables $y, \ldots, z$.
Let $k$ denote the residue of $R$ and let $K$ denote the residue field of 
the order valuation ring of $R$.
Let $\xi_i$ denote the image of $\frac{w_i}{w_{d-1}}$ in the field $K$, 
and let $L = k (\xi_1, \ldots, \xi_{d-2} )$,
so that the transcendence degree of $L$ over $k$ is $d - 2$.
Then $L$ contains the image of $\frac{\alpha}{\beta}$ in $K$,
for all elements $\alpha, \beta \in I$ of minimal order.
Theorem~\ref{3.1} implies that  $\ord_R$ is not a Rees valuation of $I$.

To see that item 4 implies item 3,  assume  there exists an  element 
$y_1^{\alpha_y} \cdots z_1^{\alpha_z} \in \Delta (I_1)$ such 
that $\alpha_y + \ldots + \alpha_z < \delta$.
By Equation~\ref{5.129}, the element 
$x^{\delta  - (\alpha_y + \ldots + \alpha_z)} y^{\alpha_y} \cdots z^{\alpha_z}$ is 
in $I$, where $\delta  - (\alpha_y + \ldots + \alpha_z) > 0$.
Since $I$ contains an element of minimal order divisible by $x$,   
Proposition~\ref{contracted} implies that 
$\ord_{R}$ is a Rees valuation of $I$.  
\end{proof}

Example~\ref{3not1} demonstrates that items 3 and 4 of  Theorem~\ref{shapeofideals} do not
in general imply items 1 and 2.

\begin{example}   \label{3not1}
Assume the notation of Setting~\ref{Setting4} with $d = 3$.
Let $I_1$ be 
the complete  $\m_1$-primary 
ideal $(x, y_1^2, y_1 z_1, z_1^2)R_1$.
Then $\CIT (I_1) = (x^3, y^2, y z, z^2, x^2 y, x^2 z)$.   The ideal 
$I_1$ satisfies items 3 and 4, but not items 1 and 2 of Theorem~\ref{shapeofideals}.
\end{example}

The ideal $I_1$ of Example~\ref{3not1} is not finitely  supported.   In connection with  
Theorem~\ref{shapeofideals},  we ask:

\begin{questions}  ~ 
\begin{enumerate}
\item  
If the complete monomial ideal $I_1$ in   Theorem~\ref{shapeofideals} is 
finitely supported and 
satisfies items 3 and 4, does it also satisfy items 1 and 2?

\item If $I$ is a finitely supported $\m$-primary complete monomial ideal of order $r$, 
is $\mu (I) = \mu (\m^r)$?
\item If $I$ is an $\m$-primary complete monomial ideal of order $r$ with only finitely 
many base points in its first neighborhood, is $\mu (I) = \mu (\m^r)$?
\item Do either of the two previous questions have an affirmative answer without 
the assumption that $I$ is a monomial ideal?
\end{enumerate}
\end{questions}

Without the assumption that $I$ has finitely many base points in its first neighborhood, 
it is easy to give examples in a $3$-dimensional regular local ring $R$ of complete 
monomial ideals 
of order $1$ that require an arbitrarily large minimal number of generators.  For example, 
for each positive integer $k$, 
the ideal $x R + (y, z)^{k} R$  is a complete monomial ideal that requires $k+2$ 
generators. 

In  Theorem~\ref{compat} we give necessary and sufficient conditions for the 
complete inverse transform of a $*$-product  of monomial ideals to be the $*$-product 
of the complete
inverse transforms of the factors. 

\begin{theorem}\label{compat} 
Assume the notation of Setting~\ref{Setting4}.
Let $I_1$ and  $ I'_1$ be complete $\m_1$-primary monomial ideals of $R_1$.
Let $y_1^{n_y}, \ldots, z_1^{n_z} \in \Delta (I_1)$  and  $y_1^{n'_y}, \ldots, z_1^{n'_z} \in \Delta (I'_1)$.  
Let 
$$
n ~: =  ~\max \{n_y, \ldots, n_z\} + ~\max \{n'_y, \ldots, n'_z\} ~ - ~\max \{n_y + n'_y, \ldots, n_z + n'_z\}.
$$
Then  $n$ is a nonnegative integer and we have 
$$
\CIT (I_1) * \CIT (I'_1)~ =   ~\m^{n} * \CIT (I_1 * I'_1).
$$
\end{theorem}

\begin{proof}
Lemma 2.3 of ~\cite{L}  implies that 
$$
\CIT (I_1) * \CIT (I'_1) ~ =  ~ \m^{k} * \CIT (I_1 * I'_1)
$$
where
$$ k~ = ~ \ord_{R} (\CIT (I_1)) + ~\ord_{R} (\CIT (I'_1)) -  ~\ord_{R} (\CIT (I_1 * I'_1))
$$
is  a nonnegative integer.  
 Lemma~\ref{5.14} implies that $k = n$.
\end{proof}

The following corollary is immediate.

\begin{corollary}
Assume the notation of Setting~\ref{Setting4}.
Let  $I_1$ be a complete $\m_1$-primary 
monomial ideal of $R_1$. With all products taken to be $*$-products, we have
\begin{enumerate}
\item For $k \ge 0$, $\CIT (I_1^k) = \CIT (I_1)^{k}$.
\item For $k \ge 0$, $\CIT (\m_1^k I_1) = \CIT (\m_1)^k \CIT (I_1)$.
\end{enumerate}
\end{corollary}

\begin{example}\label{simple}
Assume the notation of Setting~\ref{Setting4} with $d = 3$.
Consider the ideals 
$$
I_1~ =  ~(x, y_1^2, z_1) \qquad  I'_1~ =  ~(x, y_1, z_1^2) \quad \text { and } \quad
I_1I_1' ~= ~  (x^2,  xy_1, xz_1,  y_1z_1, y_1^3, z_1^3)
$$
 of $R_1$.
Let $I = \CIT(I_1)$ and $I' =  \CIT (I'_1)$.   We have 
$$
I ~=  ~(x z, y^2, y z, z^2) +~ \m^3   \qquad \text{ and } \qquad   I'~ =  ~(x y, y^2, y z, z^2) + ~\m^3,
$$
while 
$$
K~ := ~ \CIT (I_1 I'_1) ~ =  ~  (x^5, x^3 y, x^3 z, x y z) + (y, z)^3.
$$
 Theorem~\ref{compat}  implies that  $I * I' = \m * K$.   The ideals $I$ and $I'$ 
are special $*$-simple complete
 ideals each having three base points,   the first two base points being $R$ and $R_1$. 
The third base point for $I$ and $I'$ are
 $$ 
 R_2 ~=   ~R_1[\frac{x}{y_1}, \frac{z_1}{y_1}]_{(y_1,  \frac{x}{y_1}, \frac{z_1}{y_1})}  \quad \text{ and } \quad  R_2'~ =  ~
 R_1[\frac{x}{z_1}, \frac{y_1}{z_1}]_{(z_1,  \frac{x}{z_1}, \frac{y_1}{z_1})} 
 $$
respectively. Since $\m$ is clearly special $*$-simple, the expression $\m * K = I * I'$ 
is the unique factorization of $K$
as a product of special $*$-simple ideals.
The ideal $K$   has four base points $R, R_1, R_2, R_2'$ and has
point basis $3,2,1,1$.   We prove that $K$ is $*$-simple.  The points $R_2$ and $R_2'$ 
are maximal base points of $K$.  
Assume that  $K = L *W$ is a   nontrivial  $*$-factorization.
Since $K$ is a complete inverse transform, neither $L$ nor $W$ is  a power of $\m$, 
so both $L$ and $W$ have $R_1$ as a base point.
We first show that  neither $L$ nor $W$ has  both  $R_2$ and $R_2'$  as base points.    
Assume  by way of contradiction that   $L$ has both $R_2$ and $R_2'$ as base points.  Then 
$W$ has neither $R_2$ nor $R_2'$ as  a base point.
We have  $ K^{R_1} = L^{R_1}*W^{R_1}$.  
Since $L^{R_1}$ has two maximal base points,    Fact~\ref{order1} implies that  
$L^{R_1}$ has order at least $2$. 
Thus $L^{R_1} = K^{R_1}$ and $W^{R_1} = R_1$, a contradiction to the assumption  
that $R_1$ is a base point of $W$.
Thus $K = L*W$ implies that each of $L$ and $W$ contains precisely one
of the two maximal base points $R_2$ and $R_2'$.
Hence the base points of $L$ and $W$ are linearly ordered.  Theorem~\ref{nonegativeexpo} 
implies that the factorizations of $L$ and $W$ as a $*$-product of special $*$-simple ideals 
involve no negative exponents. 
Since $I$ is the special $*$-simple ideal $P_{RR_2}$  and $I'$ is the special $*$-simple 
ideal $P_{RR_2'}$, 
Remark~\ref{factprop} implies that   $L \subseteq I$  and $W \subseteq I'$ and 
therefore that $K$ is contained in 
the $I * I'$.  This contradicts the fact that $K$ has order 3 and $I * I'$ has order 4.  
We conclude that  $K$ is $*$-simple.
\end{example}

\section{Sequences of local monomial quadratic transformations}  \label{cc5}

\begin{setting}\label{5.3}
Let $(R, \m)$ be an equicharacteristic  $d$-dimensional regular local ring  
and fix a regular system of parameters   $x, y, \ldots, z$   for  $R$.     
Let $n \geq 2$ be an integer, and let  
\begin{equation} \label{5.31}
R:=R_0~\subset~R_1~\subset~R_2~ \subset~\cdots~\subset~ R_{n-1}~\subset~R_n
\end{equation} 
be a {\bf sequence of  $d$-dimensional  local monomial quadratic transformations} in the sense that 
the fixed regular system of parameters for $\m_{i+1}$ is determined inductively from $\m_i$ in the following manner:

Fix a regular system of parameters $\m_i:=(x_i,~y_i,~\ldots,~z_i)R_i$ for $i$ with $0\leq i \leq n$ as defined inductively.
Then we say that $R_{i+1}$ is a {\bf local monomial quadratic transformation} of $R_i$ in the $x$-direction
if 
$$
R_{i+1}=R_i\big[\frac{\m_i}{x_i}\big]_{(x_i,~\frac{y_i}{x_i},~\ldots,~\frac{z_i}{x_i})},
\quad \m_{i+1}=(x_{i+1},~y_{i+1},~ \ldots, ~z_{i+1})R_{i+1},
$$
$$
\text{where}\quad x_{i+1}:=x_i,
\quad y_{i+1}:=\frac{y_i}{x_i},
\quad \ldots,
\quad z_{i+1}:=\frac{z_i}{x_i}.
$$
Similarly, 
we say that $R_{i+1}$ is a {\bf local monomial 
quadratic transformation } of $R_i$ in the $y$-direction, $\ldots$, or $z$-direction
if 
$$
R_{i+1}=R_i\big[\frac{\m_i}{y_i}\big]_{(\frac{x_i}{y_i},~y_i,~\ldots,~\frac{z_i}{y_i})},
\quad \ldots,
\quad \text{or}\quad
R_{i+1}=R_i\big[\frac{\m_i}{z_i}\big]_{(\frac{x_i}{z_i},~\frac{y_i}{z_i},~\ldots,~z_i)}.
$$
Let $V_i$ denote the  order valuation ring of $R_i$, for each $i$ with $0 \le i \le n$. 
For each integer $i$ with $1 \le i \le n$ and each $\m_i$-primary 
monomial ideal $I$,  let $\CIT(I)$ denote the complete inverse transform of $I$ in $R_{i-1}$. 
\end{setting}

\begin{remark}\label{monomial valuations}
Assuming notation as in Setting~\ref{5.3},    Proposition~\ref{contract7} implies that the $V_i$-ideals in $R$ are all monomial ideals.
\end{remark}

\begin{lemma}  \label{5.40}
Assume notation as in Setting~\ref{5.3}, and let $n = 2$.
Let $I_1, \ldots, I_m$ be $\m_2$-primary complete  monomial  ideals of $R_2$.   With all products  taken to be
 $*$-products,  for each  $k  \in \N_0$    we have
$$\CIT \left( \m_1^k \prod_{i=1}^{m} \CIT (I_i) \right) ~ =  ~\CIT (\m_1^k) \prod_{i=1}^{m} \CIT ( \CIT (I_i) )$$
\end{lemma}

\begin{proof}   
Without loss of generality, we may assume that the local monomial quadratic transformation from $R_1$ to $R_2$ is in the
$x$-direction.       If $\ord_{R_1} \CIT(I_i) = e_i$,  then Remark~\ref{compat2} 
and Proposition~\ref{eq9}   imply that $w_1^{e_i} $ is a minimal 
generator for $\CIT(I_i)$ for each of the elements $w_1$  other than $x_1$ in the fixed 
regular system of parameters for $R_1$.
The assertion now follows from  Theorem~\ref{compat}.
\end{proof}  

\begin{theorem}  \label{nonegativeexpo}
Assume notation  as in Setting~\ref{5.3} and let  $I $ be a complete $\m$-primary  
monomial ideal   of $R$.  
If the  base points  of $I$ are a subset of  $\{R_0, \ldots, R_n \}$,  then 
the  unique factorization  of $I$  as a product of special  $*$-simple ideals 
involves  no negative exponents.
\end{theorem}

\begin{proof}
We  use  induction  on $n$.
If  $n = 0$, the only base point of $I$ is $R$, so $I$ is a power of $\m$ and the assertion holds.

Assume $n > 0$ and  the assertion  is true for $n - 1$.
Let $J_1 = I^{R_1}$ be the complete transform of $I$ in $R_1$, and let $J := \CIT (J_1)$.
 Lipman's unique factorization theorem implies that  there exists  
a unique set of integers $k_i$  such that  
$$
\left( \prod_{k_i < 0} P_{R_0 R_i}^{- k_i} \right) I ~=  ~\prod_{k_i > 0} P_{R_0 R_i}^{k_i},
$$ 
where all products are $*$-products. Taking complete transform to $R_1$ of the 
ideals on both sides of this equation, we obtain
$$  \left( \prod_{i \ge 1, k_i < 0} P_{R_1 R_i}^{- k_i} \right) J_1~ = ~ \prod_{i \ge 1, k_i > 0} P_{R_1 R_i}^{k_i}.
$$
By the induction hypothesis  $k_i \ge 0$     for $1 \le i \le n$.   Hence
$$
J_1 = \prod_{i \ge 1, k_i > 0} P_{R_1 R_i}^{k_i}.
$$
Taking complete inverse transform and using  Lemma~\ref{5.40} gives
$$
J ~ =   ~ \prod_{i=1}^{n} P_{R_0 R_i}^{k_i}.
$$
Since $J^{R_1} = I^{R_1}$ and $R_1$ is the only base point of $I$ or $J$ in 
the first neighborhood of $R$,  Lemma~2.3 of \cite{L}  implies that  
there exists an integer $n \ge 0$ 
such that $I = \m^{n} * J$.    It follows that  $n = k_0$    and  $I$ 
has no negative exponents in its unique factorization 
as a product of special  $*$-simple ideals.
\end{proof}

\begin{corollary}  \label{simpleisspecial}
Assume  notation as in Setting~\ref{5.3}.   If  $I$ is an $\m$-primary   
finitely supported $*$-simple complete 
 monomial ideal and the base points of $I$ are linearly ordered, 
then  $I$ is a special $*$-simple ideal.
\end{corollary}

\begin{proof}  We may assume the base points of $I$ are $\{R_0, \ldots, R_n\}$ 
as in Setting~\ref{5.3}, and apply Theorem~\ref{nonegativeexpo}.
\end{proof}

\begin{theorem}\label{5.41}
Assume notation  as in Setting~\ref{5.3}, and let $n = 2$.
Let $I_2 \subset R_2$ be an $\m_2$-primary complete monomial ideal,
and let $I_1 = \CIT (I_2)$ in $R_1$ and $I_0 = \CIT (I_1)$ in $R_0$.
The following are equivalent:
\begin{enumerate}
\item There is a change in direction from $R_0$ to $R_2$.
\item  $\mu(I_0) > \mu(I_1)$.
\item $\ord_{R_0} (I_0) = s$, where $s = \max \{ \ord_{R_1} (\alpha) \vert \alpha \in \Delta (I_1) \}$.
\item $\ord_{R_0} (I_0) > \ord_{R_1} (I_1)$.
\item $\ord_{R_0}$ is a Rees valuation of $I_0$.
\end{enumerate}
\end{theorem}

\begin{proof}
The proof proceeds as follows.
We  show that if there is a change in direction 
from $R_0$ to $R_2$, then items~2, 3, 4 and 5 hold, and 
if there is not a change in direction from $R_0$ to $R_2$, 
then  items~2, 3, 4  and 5 do not hold.
Thus, items~2, 3, 4 and 5 are all equivalent to item~1. We may assume 
without loss of generality that $R_1 \subset R_2$ is in the $x$-direction.

Let $x_2^{n_x}, y_2^{n_y}, \ldots, z_2^{n_z} \in \Delta (I_2)$, and let $r = \delta (I_2) = \max\{n_y, \ldots, n_z\}$.
Lemma~\ref{5.14} and Proposition~\ref{eq9} imply that 
 $x_1^{n_x + r}, y_1^{r}, \ldots, z_1^{r} \in \Delta (I_1)$ and $\ord_{R_1} (I_1) = r$.
Thus we have  $s = n_x + r > \ord_{R_1} (I_1)$.

Assume  there is no change in direction from $R_0$ to $R_2$.
That is, $R_0 \subset R_1$ is in the $x$-direction.
By Lemma~\ref{5.14}, $r = \delta (I_1)$, and $\ord_{R_0} (I_0) = \ord_{R_1} (I_1)$.
 Theorem~\ref{shapeofideals} implies that $\mu(I_0) = \mu(I_1)$ 
and  $V_0 \notin \Rees (I_0)$.
Thus items~2, 3, 4 and 5 do not hold.

Assume there is a change in direction from $R_0$ to $R_2$.
Without loss of generality, we may assume that  $R_0 \subset R_1$ is in the $y$-direction.
 Lemma~\ref{5.14}  implies that  $\delta (I_1) = n_x + r$, 
and $\ord_{R} (I_0) = n_x + r$.
Thus items~3 and 4 hold. 
Theorem~\ref{shapeofideals} implies that  $\mu(I_0) > \mu(I_1)$, and 
since $z^{r} \in \Delta (I_1)$ and $r < n_x + r = \delta (I_1)$, 
Theorem~\ref{shapeofideals}  also implies  that   
$V_0 \in \Rees (I_0)$.  Thus  items~2 and 5 hold.
\end{proof}

\begin{remark} \label{5.41r}
The integer  $s$ of item~2 of Theorem~\ref{5.41} 
is  the smallest integer $s$ such that $\m_1^s \subset   I_1$. It is
also equal to $\max\{a, b, \ldots, c\}$, where 
$x_1^{a}, y_1^{b}, \ldots, z_1^{c} \in \Delta (I_1)$.

Let $I$ be a complete $\m$-primary monomial ideal that has at most one base point in the first neighborhood of $R$.
By Remark~\ref{compat3}, this assumption on $I$ is equivalent to the assumption that $I$ is either  a power of $\m$, 
or $I$ is a power of $\m$ times $\CIT (I_1)$, 
where $I_1$ is an $\m_1$-primary  monomial  ideal of     the unique base point $R_1$  of $I$ in the first neighborhood of $R$.
Let $x^a, y^b, \ldots, z^c \in \Delta (I)$.
We associate with $I$ a pair of integers, $r = \min (a, b, \ldots, c)$ and $s = \max (a, b, \ldots, c)$.
With our assumptions on $I$, it follows that $r = \ord_{R} (I)$ and $\m^{s} \subseteq  I$.
That is, $\m^r \supseteq  I  \supseteq  \m^s$, and $r$ is the maximum integer and $s$ is the minimum integer such that these inclusions hold.
We call the integer $s$ the {\bf index} of $I$.
Even in the case where the $\m$-primary complete ideal $I$ is not a monomial ideal, we refer to the smallest integer $s$ such 
that $\m^s \subseteq  I$ as the {\bf index} of $I$.

Theorem~\ref{5.41} yields a description of how these invariants behave with respect to complete inverse transform.
In particular, let $I_1, I_0$ be as Theorem~\ref{5.41}, and let $s$ and $r$ be the index and order of $I_1$, respectively.
 If there is no change in direction from $R_0$ to $R_2$,   then the index and order of $I_0$ are $s + r$ and $r$, respectively.  
 If  there is a change in direction from $R_0$ to $R_2$,     then the index and order of $I_0$ are $s + r$ and $s$, respectively.    
\end{remark}

\begin{theorem}\label{5.43}
Assume notation  as in Setting~\ref{5.3}, and let $n = 3$.
Let $I_3 \subset R_3$ be an $\m_3$-primary complete monomial ideal,
and  define  $I_i = \CIT (I_{i+1})$ in $R_i$  for  $0 \le i \le 2$.
The following are equivalent:
\begin{enumerate}
\item There is a change in direction from $R_0$ to $R_2$ and a change 
in direction from $R_1$ to $R_3$.
\item  $\mu(I_0) > \mu(I_1) > \mu(I_2)$. 
\item $\ord_{R_0} (I_0) =  \ord_{R_1}(I_1) + \ord_{R_2}(I_2)$.
\item $\ord_{R_0}$   and $\ord_{R_1}$ are both  Rees valuations   of $I_0$.
\end{enumerate}
\end{theorem}

\begin{proof} 
Applying Theorem~\ref{5.41}, it is straightforward that items 1 and 2 are equivalent.
By Proposition~\ref{reestrans},
$\ord_{R_1} \in \Rees_{R_0} (I_0)$ if and only if $\ord_{R_1} \in \Rees_{R_1} (I_1)$.
Thus it is clear that items 1 and 4 are equivalent.
It remains to show that item 1 is equivalent to item 3.

We may assume without loss of generality that $R_2 \subset R_3$ is in the $x$-direction. 
Let $x^{a_3}, y^{b_3}, \ldots, z^{c_3} \in \Delta(I_3)$ and let 
$r_2 = \delta(I_3) = \max\{b_3, \ldots,  c_3\}$. 
Lemma~\ref{5.14} and Proposition~\ref{eq9} imply that 
 $x_2^{a_3 + r_2}, y_1^{r_2}, \ldots,  z_1^{r_2} \in \Delta (I_2)$ and $\ord_{R_2} (I_2) = r_2$.

We compute the index and order of $I_1$, and $I_0$ as in Remark~\ref{5.41r}.
Let $s$ and $r$ be the index and order of $I_2$, respectively, and notice that $s > r$.
We have the following diagram.
In the diagram, the index and order of $I_i$ are on the $i$-th level.
Going down to the left from level $i$ to level $i - 1$ indicates a change in direction from $R_{i-1}$ to $R_{i+1}$, whereas going down to the right indicates no change in direction.

$$\begin{tikzpicture}[level distance=1.5cm,
level 1/.style={sibling distance=6cm},
level 2/.style={sibling distance=3cm}]
\node (Root) {$(s, r)$}
child {
	node {$(s + r, s)$}
	child {
		node {$(2 s + r, s + r)$}
	}
	child {
		node {$(2 s + r, s)$}
	}
}
child {
	node {$(s + r, r)$}
	child {
		node {$(s + 2 r, s + r)$}
	}
	child {
		node {$(s + 2 r, r)$}
	}
};
\node [right of=Root, node distance=7cm] (Level2) {$(2)$}
child { node {$(1)$} edge from parent[draw=none]
	child { node{$(0)$} edge from parent[draw=none]}
};
\end{tikzpicture}$$

Since $s > r$,  it follows from this diagram that $\ord_{R_0} (I_0) = \ord_{R_1} (I_1) + \ord_{R_2} (I_2)$ if and only  if item~1 holds. 
\end{proof} 

Theorem~\ref{5.43}  directly  implies the following more general result that we state as Corollary~\ref{5.431}.  

\begin{corollary}  \label{5.431}
Assume notation as in Setting~\ref{5.3}, with $n \ge 3$.
Let $I_n \subset R_n$ be an $\m_n$-primary complete monomial ideal.
Define $I_i = \CIT (I_{i+1})$ in $R_i$ for each $i$ with  $0 \le i \le n - 1$.
The following are equivalent:
\begin{enumerate}
\item There is a change in direction from $R_i$ to $R_{i+2}$ for  each $i$ with $0 \le i \le n - 2$.
\item  $\mu (I_0) > \mu (I_1) > \ldots > \mu (I_{n-1})$.
\item  $\ord_{R_i} (I_i) = \ord_{R_{i+1}} (I_{i+1}) + \ord_{R_{i+2}} (I_{i+2})$ for each  $i$ with $0 \le i \le n - 3$.
\item  $\ord_{R_i}$ is a Rees valuation of $I_0$ for each $i$ with  $0 \le i \le n - 2$.
\end{enumerate}
\end{corollary}

\begin{remark}  \label{gcd one}
We describe all possible ordered pairs $(s, r)$ such that $s$ is the index and $r$ is the order of a special $*$-simple monomial ideal.
Assume the notation of Setting~\ref{5.3}.
The index and order of $P_{R_n R_n} = \m_n$ and $P_{R_{n-1} R_n}$ are $(1, 1)$ and $(2, 1)$, respectively.
For $i$ with $0 \le i \le n - 2$, if the index and order of $P_{R_{i+1} R_n}$ is $(s, r)$, then Theorem~\ref{5.43} implies the index and order of $P_{R_i R_n}$ is $(s + r, s)$ if there is a change of direction from $R_i$ to $R_{i+2}$ and $(s + r, r)$ if there is no change of direction from $R_i$ to $R_{i+2}$.
Diagram~\ref{MonomialTree} illustrates the first few levels of an infinite tree that describes this behavior, where the left path indicates a change of direction and the right path indicates no change of direction.

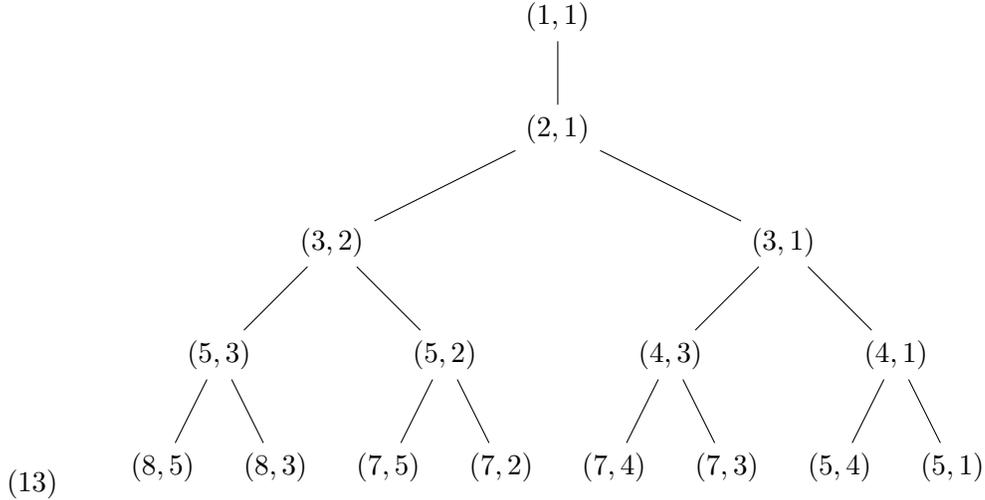
\begin{equation}\label{MonomialTree}\begin{tikzpicture}[level distance=1.5cm,
level 2/.style={sibling distance=6cm},
level 3/.style={sibling distance=3cm},
level 4/.style={sibling distance=1.5cm}]
\node (Root) {$(1, 1)$}
child {
	node {$(2, 1)$}
	child {
		node {$(3, 2)$}
		child {
			node {$(5, 3)$}
			child { node {$(8, 5)$} } 
			child { node {$(8, 3)$} }
		}
		child { 
			node {$(5, 2)$}
			child { node {$(7, 5)$} } 
			child { node {$(7, 2)$} }
		}
	}
	child {
		node {$(3, 1)$}
		child { 
			node {$(4, 3)$}
			child { node {$(7, 4)$} } 
			child { node {$(7, 3)$} }
		}
		child {
			node {$(4, 1)$}
			child { node {$(5, 4)$} } 
			child { node {$(5, 1)$} }
		}
	}
};
%\node [right of=Root, node distance=7cm] (Level2) {$(2)$}
%child { node {$(1)$} edge from parent[draw=none]
%	child { node{$(0)$} edge from parent[draw=none]}
%};
\end{tikzpicture}\end{equation}

For a vertex $(s, r)$ at a given level past the first (that is, $s \ge 2$), there are precisely two vertices at the next level adjacent to $(s, r)$, namely $(s + r, s)$ and $(s + r, r)$.
Since $\gcd (s, r) = 1$ implies that $\gcd (s + r, r) = 1$ and $\gcd (s + r, s) = 1$, and $\gcd (1, 1) = 1$,
every  ordered pair   $(s, r)$  of positive integers   that may be    realized as the index and order of 
a special $*$-simple monomial ideal     satisfies   the properties:   (i)  $s \ge r$,  and  (ii) $\gcd (s, r) = 1$.
We show in Theorem~\ref{invariants} that every pair $(s, r)$ satisfying (i) and (ii) 
is realized as the index and order of a special $*$-simple monomial ideal and 
observe uniqueness properties of this realization.
\end{remark}

\begin{proposition}\label{invariants}
Let $(s, r)$ be an ordered pair of positive integers such that $s \ge r$ and $gcd (s, r) = 1$.
Then $(s, r)$ occurs exactly once in the tree described in Diagram~\ref{MonomialTree}.
\end{proposition}

\begin{proof}
We use induction on the positive integer $s$.
The cases where $s \le 2$ are clear.

Assume $s > 2$ and that for all positive integers $s' < s$ and all ordered pairs $(s', r')$ 
that satisfy $s' \ge r'$  and  $\gcd (s', r') = 1$, the assertions of 
Proposition~\ref{invariants}.

Let $r$ be a positive integer with $s > r$ and $\gcd (s, r) = 1$.
Either we have $s - r > r$ or $s - r < r$.

\textbf{Case 1:}
Assume that $s - r > r$.
Consider the pair $(s - r, r)$.
By the induction hypothesis, the pair $(s - r, r)$ occurs exactly once in Diagram~\ref{MonomialTree}.
Passing one step down in the diagram from $(s - r, r)$ to the right gives $(s, r)$.

Suppose that $(s, r)$ occurs as the child node of some $(s', r')$.
Thus $s' + r' = s$, and either $s' = r$ or $r' = r$.
If $s' = r$, then $r' = s - r > r = s'$, which contradicts the fact that $s' \ge r'$, so it must be the case that $r' = r$.
Thus $(s', r') = (s - r, r)$.

\textbf{Case 2:}
Assume that $s - r < r$.
Similarly to the previous case, the pair $(r, s - r)$ occurs exactly once, and $(s, r)$ is obtained by passing down one step to the left.

Suppose that $(s, r)$ occurs as the child node of some $(s', r')$.
As before, $s' + r' = s$, and either $s' = r$ or $r' = r$.
If $r' = r$, then $s' = s - r < r = r'$, which contradicts the fact that $s' \ge r'$, so $s' = r$.
Thus $(s', r') = (r, s - r)$.
\end{proof}

We record in Corollaries~\ref{5.5} and \ref{5.6} implications of   Theorem~ \ref{5.43} for special $*$-simple monomial ideals.  

\begin{corollary}\label{5.5}
Assume notation  as in Setting~\ref{5.3}, and fix  $i$ with $0 \leq i \leq n-3$.
If there are two change of directions from $R_i$ to $R_{i+3}$,
then we have 
$$
\ord_{R_i}(P_{R_iR_n})=\ord_{R_{i+1}}(P_{R_{i+1}R_n})+\ord_{R_{i+2}}(P_{R_{i+2}R_n}).
$$
\end{corollary}

\begin{corollary}\label{5.6}
Assume notation  as in Setting~\ref{5.3},  and let $I := P_{R_0R_n}$ 
denote the special $*$-simple complete ideal associated to the
sequence of 
local monomial quadratic transformations of Equation~\ref{5.31}. 
Let $r_i:=\ord_{R_i}(P_{R_iR_n})$ for  each  $i$ with $0 \leq i \leq n$.   If 
for each $i$ with $0 \le i \le n-2$ 
there is  a
change of direction between $R_i$ and $R_{i+2}$,  then
\begin{enumerate}
\item
$V_i\in \Rees(P_{R_iR_n})$ for all $i=0,1,2,\ldots, n-2$.
\item
$\Rees(I)=\{V_0,~V_1,~ V_2,~ \ldots, V_{n-2},~V_n\}$.
\item
$\mathcal B(I)=\{r_0, r_1, r_2, r_3, \ldots, 13,8,5,3,2,1,1\}$ is a Fibonacci sequence.
\item
The sequence $(s_0, s_1, \ldots, 13, 8, 5, 3, 2, 1)$, where $s_i$ is the index of $P_{R_iR_n}$, is a shift of a Fibonacci sequence.
\end{enumerate}
\end{corollary}

 Example~\ref{7.9} describes the structure 
of the special $*$-simple complete ideal $P_{R_0R_4}$ in the 
case where there is a change of direction 
from $R_0$ to $R_2$  and from $R_2$ to $R_4$, but 
there is no change of direction from $R_1$ to $R_3$.

\begin{example}\label{7.9}
Let the notation be as in  Setting~\ref{5.3} with $d = 3$ , $\m = (x,y,z)R$ and  $n = 4$.  Assume 
that  the local quadratic transforms are:
$$
R  ~:=R_0  ~\subset~  {^xR_1} ~\subset~ {^{yx}R_2} ~\subset ~{^{yyx}R_3}~\subset~ {^{zyyx}R_4}. 
$$
defined by
$$
\begin{aligned}
& S_1  :=  ~R[\frac{\m}{x}],~~ N_1 := ~ (x,~ \frac{y}{x},~ \frac{z}{x})S_1, ~~R_1 := ~ (S_1)_{N_1},~~ \m_1 := ~ N_1R_1:=(x_1,~y_1,~z_1)R_1. \\
& S_{2} := ~ R_1[\frac{\m_1}{y_1}],~~ N_{2} := ~(\frac{x_1}{y_1},~ y_{1},~ \frac{z_{1}}{y_{1}})S_{2},
           ~~ R_{2} := ~ (S_{2})_{N_{2}},~~\m_{2} := ~ N_{2}R_{2}:=(x_{2},~y_{2}, ~z_{2})R_{2}.\\
& S_{3} := ~ R_2[\frac{\m_2}{y_2}],~~ N_{3} := ~(\frac{x_2}{y_2},~ y_{2},~ \frac{z_{2}}{y_{2}})S_{3},~~
            R_{3} := ~ (S_{3})_{N_{3}},~~ \m_{3} := ~ N_{3}R_{3}:=(x_{3},~y_{3}, ~z_{3})R_{3}.\\
& S_{4} := ~ R_3[\frac{\m_3}{z_3}],~~ N_{4} ~:= ~(\frac{x_3}{z_3},~ \frac{y_3}{z_3},~ z_{3})S_{4},~~ 
            R_{4} := ~ (S_{4})_{N_{4}},~~\m_{4} := ~ N_{4}R_{4}:=(x_{4},~y_{4}, ~z_{4})R_{4}.\\
\end{aligned}
$$

The sequence of special $*$-simple ideals is:
\begin{align*}
P_{R_4 R_4} ~ =&   ~(x_4,~  y_4, ~ z_4). \\
P_{R_3 R_4}  ~=&  ~(x_3, ~ y_3,  ~z_3^2). \\
P_{R_2 R_4} ~=& ~ (x_2^2,~ x_2 z_2, ~ z_2^2, ~ x_2 y_2,  ~y_2^2 z_2, ~ y_2^3). \\
P_{R_1 R_4~} =& ~ (x_1^2, ~ x_1 z_1, ~ z_1^2,  ~x_1 y_1^2,  ~y_1^3 z_1, ~ y_1^5).
\end{align*}

Then:
\begin{enumerate}
\item
Let $v_i:=\ord_{R_i}$ for each $0 \leq i\leq 4$. Then we have
\[
\begin{array}{c|c|c|c }
                     & x & y & z         \\ \hline
v_4:=\ord_{R_4}      & 6   & 8    & 11 \\ \hline
v_3:=\ord_{R_3}      &3     &4     &6   \\ \hline
v_2:=\ord_{R_2}      &2   &3  &4         \\ \hline
v_1:=\ord_{R_1}        &1  &2  &2       \\ \hline
v_0:=\ord_{R_0}       &1  &1  &1                      
\end{array}
\]
\item
The special $*$-simple complete $\m$-primary ideal $P_{R_0R_4}$ is  the ideal  $K$, where
$$
\aligned
K:&= \{\alpha \in \m~\vert~ v_4(\alpha) \geq 40=v_4(y^5)~~ \text{and}~~v_0(\alpha) \geq 5\}\\
   &= (y^5,~x^4y^2,~x^3z^2,~xy^3z,~x^5z,~x^7,~x^2yz^2,~x^2z^3,~xy^2z^2,~xyz^3,~xz^4,~\\
   &\qquad y^4z,~y^3z^2,~y^2z^3,~yz^4,~z^5,~x^4yz,~x^3y^3,~x^3y^2z,~x^2y^4,~x^6y)R.
\endaligned
$$
%This is another description for $K$ we want to justify.

\item
$\mathcal BP(P_{R_0R_4})=\{R_0, R_1, R_2, R_3, R_4\}.$
\item
$\mathcal B(P_{R_0R_4})=\{5,2,2,1,1\}.$
\item
$(s_0, s_1, s_2, s_3, s_4) = (7, 5, 3, 2, 1)$, where $s_i$ is the index of $P_{R_i R_4}$.
\item
The set of Rees valuations of $P_{R_0R_4}$ is $\{\ord_{R_0},~\ord_{R_2},~\ord_{R_4}\}$.
\end{enumerate}
\end{example}

\begin{proof}
Item~1 is clear. 
Item~2 follows from Remark~\ref{monomial valuations} and Lemma~\ref{mingens}.
Item~3  is clear. 
Theorem~\ref{5.41} implies items 4, 5, and 6.
\end{proof}

\begin{remark}
In Example~\ref{7.9}, the ideal  $K = P_{R_0R_4}$ has three Rees valuations 
with   one of these Rees
valuations,  $\ord_{R_2}$,   redundant in the representation of $K$ as the intersection 
of valuation ideals corresponding to its Rees valuations.
\end{remark}

\begin{question}    \label{5.11q}    With $(R,\m)$ as in Definition~\ref{5.1},  let $I$ be a finitely supported monomial $\m$-primary ideal.
\begin{enumerate}
\item If $I$ is complete, does it follow that $\m I$ is complete?  
\item If $I$ is contracted from $\Proj R [\m t]$, does it follow that $\m I$ is   contracted   from $\Proj R [\m t]$?
\end{enumerate}
\end{question}

\begin{remark}  \label{5.11r}    If the ideal  $I$ 
in Question~\ref{5.11q}  has only one base point in the first neighborhood of $R$, and 
if $I$ is contracted from $\Proj R [\m t]$ ,   
then  $\m I$ is contracted from $\Proj R [\m t]$.   If, for example,   
the first neighborhood of $R$
is in the  $x$-direction, then $I$ is contracted from $S = R [\frac{\m}{x}]$.   
Hence  for $f \in R$ if  $x f \in I$  then also  
$y f \in I$, $\ldots $,  $z f \in I$.  
This  same condition holds for $\m I$, so $\m I$ is  contracted from $S$. 
Since $\m I$ has the same base points in the first 
neighborhood of $R$ as $I$,  it follows that $\m I$ is contracted from $\Proj R[\m t]$.
A simple induction argument implies that for every positive integer $n$ the 
ideal  $\m^n I$ is contracted from $\Proj R [\m t]$.
\end{remark}

Without the assumption that the complete $\m$-primary monomial ideal is finitely supported,  
the ideal   $\m I$ may fail to be complete as 
we demonstrate in  Example~\ref{notfinsup}.   Example~\ref{notfinsup} is a 
modification of the example given in  \cite[Exercise 1.15]{SH}.

\begin{example}\label{notfinsup}
Let  $(R,\m)$ be  in Definition~\ref{5.1} with $d = 3$  and $\m = (x, y, z)R$.
Let $I   =  (x^{12}, y^7z^5)R +  \m^{13}$.   Then  
$I$ is an integrally closed ideal.
The integral closure of $\m I$ contains monomials of order 13  that  are not in $\m I$.  
Using that $I$ is the integral closure of $(x^{12}, y^7 z^5, y^{13}, z^{13})R$, 
we see that the Rees valuations of $I$ are  the monomial valuations  $v$ and $w$ where: 
	$$v (x) = 91, \quad v (y) = 96, \quad v (z) = 84$$
	$$w (x) = 65, \quad w (y) = 60, \quad w (z) = 72$$
Examples of monomials integral over $\m I$ are $x^6 y^4 z^3$ and $x^2 y^6 z^5$.
Let $J$ denote the integral closure of $\m I$.  Then  $J = \m * I$.
Since $I$ and $(J : \m)$ are complete ideals, and for each valuation $v'$ dominating $R$ 
we have $v' (J) = v '(\m) + v' (I)$ and $v' (J : \m) \ge v '(J) - v' (\m) = v '(I)$, it follows that $J : \m \subseteq I$ .
We clearly have $I \subseteq J : \m$, so $I = J : \m$ and hence $J = \m * (J : \m)$.
Thus we have constructed a complete monomial ideal $J$ such that $\m * (J : \m) = J$, but $\m (J : \m) \ne J$.
\end{example}

\section{The Monomial Condition for Transforms} 

\begin{discussion}  \label{13}
Let $(R, \m)$ be a $d$-dimensional equicharacteristic regular local ring.  
In Section~\ref{c5} we define monomial ideals with respect to a fixed regular
system of parameters for $R$. We then examine properties of these monomial ideals
with respect to monomial local quadratic transforms and inverse transforms.
With a fixed regular system of parameters for $R$, in Section~\ref{cc5}  we 
consider a finite sequence   $(R_i,\m_i)$ of  local monomial quadratic transformation of $R$,
where the variables for $R_{i+1}$ are determined by the variables for  $R_i$ as in Setting~\ref{5.3}.
For such a sequence as in Setting~\ref{5.3}, the special $*$-simple ideal $P_{R_0R_n}$ is then
a monomial ideal, and Corollary~\ref{5.6} describes properties of the index and order of these monomial ideals.  

 In 
this connection, it is natural to ask:  let $n   \in \N_0$, and  let 
$(R_i,\m_i)$ be a sequence of local quadratic transformations  
\begin{equation}  \label{eq13}  
R ~:= ~ R_0 ~\subset ~ R_1 ~\subset ~  \ldots \subset ~ R_{n-1} ~ \subset  ~R_{n}
\end{equation}
 such that $R_0 / \m_0 = R_{n} / \m_{n}$.  Under what conditions does there exist a 
 regular system of parameters for $R$ such that with respect to this system of parameters the 
 local quadratic transformations in Equation~\ref{eq13} are monomial?    It is clear that for $n = 1$,  so a local quadratic 
 transformation $R  \subset R_1$  with $R/\m = R_1/\m_1$, 
 the answer is that always such a regular system of parameters for $R$  can be found.   Theorem~\ref{is monomial}
 implies that the answer is also affirmative for $n = 2$, while Example~\ref{not monomial} shows that for $n= 3$
  the answer in general is negative.   
\end{discussion}

 Theorem~\ref{is monomial}   gives sufficient conditions in order that the  sequence in Equation~\ref{eq13} be  monomial
with respect to some regular system of parameters for $R$.

\begin{theorem}  \label{is monomial}  Let $(R,\m)$ be a $d$-dimensional  equicharacteristic regular local ring
and let  $(R_i,\m_i)$ with $0 \le i \le n + 1$ be a sequence of local quadratic transformations with $R = R_0$ such that $R_0 / \m_0 = R_{n+1} / \m_{n+1}$.
If there is no change of direction from $R_0$ to $R_{n}$,  then 
there exists a regular system of parameters   for $R$ such that  with respect to these parameters,   
the sequence  $R_0    \subset \dots \subset R_{n+1}$ 
is monomial as in Setting~\ref{5.3}.
\end{theorem}

\begin{proof}
If there is  no change in direction from $R_{n-1}$ to $R_{n+1}$, the conclusion follows 
from Proposition~\ref{7.5}.
Assume  there is a change in direction from $R_{n-1}$ to $R_{n+1}$.
Proposition~\ref{7.5} applied to the sequence from $R_0$ to $R_{n}$ implies  there exists a regular system of parameters
$(x, y, \ldots,   w, \ldots, z)$  for $R$  such that 
$\m_i = (x, \frac{y}{x^i},  \ldots, \frac{w}{x^i}, \ldots, \frac{z}{x^i})$,  for  each $i$ 
with  $1 \le i \le n $.  
Let $x_n = x$, $y_n = \frac{y}{x^n}$, $\ldots$, $w_n = \frac{w}{x^n}$, $\ldots$, $z_n = \frac{z}{x^n}$.

We may assume  without loss of generality that $R_{n+1}$ is a localization 
of $R_{n} [\frac{\m_{n}}{y_{n}}]$.
Since $R_0 / \m_0 = R_{n+1} / \m_{n+1}$, we have 
$$
	\m_{n+1} ~ =  ~( \frac{x_{n}}{y_{n}} - c_{x}, ~ y_{n}, \ldots, \frac{w_n}{y_n} - c_w, \ldots, \frac{z_{n}}{y_{n}} - c_{z})R_{n+1},
$$
where for each variable $w$   the element  $c_{w} \in R_0$.
Since there is a change of direction from $R_{n-1}$ to $R_{n+1}$,
we have  $\frac{x_{n}}{y_{n}} \in \m_{n+1}$.  Thus we must have $c_x \in \m$,  
and we may assume  $c_{x} = 0$.

For each variable $w$  other than  $x$ and  $y$, we set $w' = w - c_{w} y$.    We have 
$$
\m  ~=~   (x,y, \ldots, w', \ldots, z')R \quad \text{ and } \quad  \m_i~=~ (x, \frac{y}{x^i}, \ldots, \frac{w'}{x^i}, \ldots, \frac{z'}{x^i})R_i,
$$
for each $i$ with $1 \le i \le n$.
Notice that $w'   = x^n(w_n - c_wy_n)$.   Thus
$$
	\m_{n+1} ~ =  ~( \frac{x_{n}}{y_{n}} , ~ y_{n}, \ldots, \frac{w'}{x^ny_n}, \ldots, \frac{z'}{x^ny_{n}})R_{n+1}.
$$
Hence    $R_0    \subset \dots \subset R_{n+1}$  is monomial with respect 
to $(x, y, \ldots, w', \ldots, z')$.
\end{proof}

Theorem~\ref{is monomial}  together with  Theorem~\ref{5.41} yield 
the following  description of  the special $*$-simple
ideal $P_{R_0R_{n+1}}$ in the case  where 
there is no change of direction from $R_0$ to $R_{n}$,
and there is a change of direction 
from $R_{n-1}$ to $R_{n+1}$.

\begin{corollary}\label{invar4}
 Let $(R,\m)$ and $(R_i,\m_i)$ be as in  Theorem~\ref{is monomial}.  
If there is no change of direction from $R_0$ to $R_{n-1}$, and there is a change of 
direction from $R_{n-1}$  to $ R_{n+1}$, then the special $*$-simple 
ideal $I := P_{RR_{n+1}}$ has
the following properties: 
\begin{enumerate}
\item $\mu (I) = \mu (\m^2) = \binom{d+1}{2}$.
\item 
$\mathcal B(I)=\{2, \ldots, 2,1,1\}$. 
\item The index and order of $I$ are $(2 n + 1, 2)$.
\item  The Rees valuations of $I$ are  the order valuations of $R_{n-1}$ and $R_{n+1}$.
\end{enumerate}

\end{corollary}

\begin{proof}
Items 1, 2, and 3 follow from Theorem~\ref{5.41}.
Item 4 follows from Theorem~\ref{5.41}, Remark~\ref{7.1.5}, and Remark~\ref{7.2}.
\end{proof}

\begin{example}  Assume the notation of Corollary~\ref{invar4} and that 
$\dim R = 3$ with $\m = (x, y, z)$.  Then the special 
$*$-simple ideal $I = P_{RR_{n+1}}$  has minimal monomial generators,
	$$I = (y^2, ~ y z, ~ z^2, ~ x^{n} z, ~ x^{n+1} y, ~ x^{2n + 1}).$$
\end{example}

Theorem~\ref{is monomial2} gives other sufficient conditions in order
that the  sequence in Equation~\ref{eq13} be  monomial
with respect to some regular system of parameters for $R$.

\begin{theorem} \label{is monomial2}  Let $(R,\m)$ be a $d$-dimensional 
equicharacteristic regular local ring, and let 
$R = R_0 \subset R_1 \subset R_2 \subset R_3$ be a sequence of local 
quadratic transforms such that $R/\m = R_3/\m_3$.
If there is a change of direction from $R$ to $R_{2}$ and a change of 
direction from $R_1$ to $R_3$, then
there exists a regular system of parameters   for $R$ such that  
with respect to these parameters,   
the sequence  $R_0    \subset \dots \subset R_{3}$ 
is monomial as in Setting~\ref{5.3}.
\end{theorem} 

\begin{proof} Theorem~\ref{is monomial} implies that the sequence $R \subset R_1 \subset R_2$
is monomial.  Hence there exists a regular system of parameters $x, y,  \ldots,  z$ for R
such that $R \subset R_1 \subset R_2$ is monomial with 
$ \m_1 ~ =  ~ (x, \frac{y}{x}, \ldots,  \frac{w}{x}, \ldots,  \frac{z}{x})R_1$   and   
$ \m_2 ~ = ~ (\frac{x^2}{y}, \frac{y}{x}, \ldots, \frac{w}{y},  \ldots, \frac{z}{y})R_2$,
that is, the extension $R  \subset R_1$ is monomial in the $x$-direction and $R_1 \subset R_2$ is
monomial in the $y$-direction.  
Let $x_2 =   \frac{x^2}{y}, ~ y_2 =  \frac{y}{x}, \ldots, w_2 =  \frac{w}{y},  \ldots, z_2 = \frac{z}{y}$.
Since there is a change of direction from $R_1$ to $R_3$,   the affine 
component $R_2[\frac{\m_2}{y_2}]$ of the
blowup of $\m_2$ is not contained in $R_3$.  Hence $R_3$ is contained in and thus 
is a localization of at least one of the other affine
components $R_2[\frac{\m_2}{x_2}]$, $ \ldots,  R_2[\frac{\m_2}{w_2}]$, $ \ldots,  
R_2[\frac{\m_2}{z_2}]$.    
If $R_2[\frac{\m_2}{x_2}]$ is 
contained in $R_3$,  then $\m_3 =  
(x_2, \frac{y_2}{x_2} - c_y, \ldots, \frac{w_2}{x_2} - c_w, \ldots, \frac{z_2}{x_2} - c_z)R_3$,
where the elements $c_y, \ldots, c_w, \ldots, c_z$  may be taken to be in $R$ since $R/\m = R_3/\m_3$.  
Since there is a change of direction from $R_{1}$ to $R_{3}$,
we have  $\frac{y_{2}}{x_{2}} \in \m_{3}$.  Thus we must have $c_y \in \m$,  
and we may assume  $c_{y} = 0$.  
Define $w' = w - c_wx^2$ for each variable $w$ other than $x$ and $y$.  Then  $x, y, \dots w', \ldots, z'$ is a 
regular system of parameters for $R$ and we have  
$$
\frac{w_2}{x_2} - c_w ~ =  ~ \frac{w}{x^2} - c_w ~ = ~ \frac{w - c_wx^2}{x^2} ~=~  \frac{w'}{x^2}
$$
for each variable $w$ other than $x$ and $y$. Thus the sequence $R$ to $R_3$ is monomial with
respect to the regular system of parameters  $x, y, \dots w', \ldots, z'$ for $R$.    

It remains to consider the case where  $R_2[\frac{\m_2}{x_2}]$ is not
contained in $R_3$.  Then $R_3$ is a localization of $R_2[\frac{\m_2}{w_2}]$ for some $w_2$.
We may assume $R_3$ is a localization of $R_2[\frac{\m_2}{z_2}]$.
Thus $\m_3 =  
(\frac{x_2}{z_2} - c_x, \frac{y_2}{z_2} - c_y, \ldots, \frac{w_2}{z_2} - c_w, \ldots, z_2)R_3$.  
As in the previous case, we may assume $c_y = 0$.
Because $R_3$ does not contain 
$R_2 [\frac{\m_2}{x_2}]$, we have  $\frac{x_2}{z_2} \in \m_3$, so we may assume $c_x = 0$.
For each variable $w$ other than $x, y,$ and  $z$, we define $w' = w - c_w z$.
Then $x, y, \ldots, w', \ldots z$ is a regular system of parameters for $R$.
We have 
	$$\frac{w_2}{z_2} - c_w = \frac{w}{z} - c_w = \frac{w - c_w z}{z} = \frac{w'}{z}.$$
Therefore the  sequence $R$ to $R_3$ is monomial with respect to the 
regular system of parameters $x, y, \ldots, w', \ldots z$ for $R$.
\end{proof}

We observe a relationship between proximate points and change of direction.
We recall the following definition.

\begin{definition}\label{6.1}
Let $\alpha \subsetneq \beta$ be a birational extension of $d$-dimensional regular local rings.
Then $\beta$ is said to be {\bf proximate} to $\alpha$ if $\beta \subseteq V_{\alpha}$,
where $V_{\alpha}$ denotes the order valuation ring of $\alpha$.
\end{definition}

\begin{proposition}\label{6.3}
Assume notation as in Discussion~\ref{13} with $n = 2$, and let $V$ denote the order
valuation ring for $R$. 
The following are equivalent:
\begin{enumerate}
\item There is a change of direction from $R_0$ to $R_2$.
\item $R_2$ is proximate to $R_0$.
\end{enumerate}
\end{proposition}

\begin{proof}
By Theorem~\ref{is monomial}, we may assume that the sequence of local 
quadratic transforms from $R_0$ to $R_2$ 
is a monomial sequence with respect to the 
regular system of parameters $x, y, \ldots, z$ for $R$. 

To show 2 implies 1, assume there is no change of direction from $R_0$ to $R_2$.
Then we may assume $R_0$ to $R_1$ and $R_1$ to $R_2$ are in the $x$-direction.
Thus $\frac{y}{x} \in \m_1$,  and  $\frac{y}{x^2} \in \m_2$.
Since $\frac{y}{x^2} \notin V$, we have  $R_2 \nsubseteq V$, so $R_2$ 
 is not proximate to $R_0$.

To show  1 implies 2, assume  there is a change of direction from $R_0$ to $R_2$.
Without loss of generality, $R_0$ to $R_1$ is in the $x$-direction 
and $R_1$ to $R_2$ is in the $y$-direction.
Let  $w$ denote  any  one of the  elements in the fixed regular
system of parameters for $R$ other than $x$ or $y$.  Then
$$
\begin{aligned}
\m_1 &= (x_1, y_1, \ldots, w_1, \ldots, z_1), 
\quad    x_1 = x, \quad y_1 = \frac{y}{x}, \quad w_1 = \frac{w}{x}  \\
\m_2 &= (x_2, y_2, \ldots, w_2), \ldots, z_2),
\quad  y_2 = y_1 = \frac{y}{x}, \quad x_2 = \frac{x_1}{y_1} = \frac{x^2}{y}, \quad w_2 = \frac{w_1}{y_1} = \frac{w}{y}
\end{aligned}
$$
We have  $V = (R_1)_{x R_1}$.  
Consider the ring $S = R_1 [\frac{\m_1}{y_1}]$.  Then $S \subset V$  and $V = S_{\p}$, where $\p$ is a 
height-one prime of $S$.  Since $S[\frac{1}{y_1}] = R_1[\frac{1}{y_1}]$, the 
height-one  primes of $S$ not containing $y_1$ are in one-to-one  correspondence  with the height-one  primes of $R_1$ 
not containing $y_1$.  Each of the rings $R_1$ and $S$ has precisely one height-one prime containing $y_1$, namely $y_1R_!$ and
$y_1S$.  Since $x_2$ has positive $V$-value and is not in any of the other height-one primes of $S$, we have 
$x_2 S = \p$  and  $V = (R_2)_{x_2R_2}$.   so $R_2$ is proximate to $R_0$.  
\end{proof}

The invariants $(s, r)$   defined in  Remark~\ref{5.41r}   of a special $*$-simple ideal     need not be relatively prime  
if the ideal is not monomial.  We demonstrate  this in Example~\ref{not monomial}.

\begin{example}  \label{not monomial}
Let $(R, \m)$ be an  equicharacteristic  $3$-dimensional regular local ring with $\m = (x, y, z)R$.
The ideal 
	$$I = (y^2 - x^3, ~ x^2 y, ~ x y^2, x z, ~ y z, ~ z^2)R$$
is readily seen to be special $*$-simple with base points
	$$R ~ =  ~R_0 ~ \subset ~ R_1~ \subset ~ R_2  ~\subset ~ R_3,$$
where 
$R_1$ and $R_2 $ are obtained from $R_0$ and $R_1$ by taking the local monomial quadratic transformations  
in the $x$-direction and $y$-direction, respectively.
Thus 
$R_1  = R[\frac{\m}{x}]_{(x, \frac{y}{x}, \frac{z}{x})}$, with $x_1 = x$,  $y_1 = \frac{y}{x}$ and $z_1 = \frac{z}{x}$,   and 
$R_2 = R_1[\frac{\m_1}{y_1}]_{(\frac{x_1}{y_1}, y_1, \frac{z_1}{y_1})}$ with $x_2 = \frac{x_1}{y_1}$,  $y_2 = y_1$ and 
$z_2 = \frac{z_1}{y_1}$.  The local quadratic transformation $R_2$ to  $R_3$ is 
$$R_3 = R_2 \big[ \frac{\m_2}{y_2} \big]_{(\frac{x_2}{y_2} - 1,~ y_2, ~ \frac{z_2}{y_2})}.$$
Thus $x_3 = \frac{x_2-y_2}{y_2}$,  $y_3 = y_2$ and $z_3 = \frac{z_2}{y_2}$.    
Let  $\nu$  denote the order valuation of $R_3$ and $V$ the corresponding valuation ring.  It is 
readily seen that 
	$$\nu (x) = 2, \quad \nu (y) = 3, \quad \nu (y^2 - x^3) = 7, \quad \nu (z) = 5.$$
Further, $I = I V \cap R$, so $I$ is a valuation ideal  and $\nu (I) = 7$.  The ideal $I$ has
order $2$,  and $x^3 \not\in I$ implies $\m^3$ is not contained in $I$. However,  we have  $\m^4 \subset I$.  
Thus $I$ has order $2$ and index $4$.  The point basis $\mathcal B(I) = \{2, 1, 1, 1\}$ is not  the 
point basis of a special $*$-simple monomial ideal.  
As noted in 
Remark~\ref{gcd one}, the invariants $(s,r)$ of a special $*$-simple monomial ideal are relatively prime.  
Therefore   there does not exist a regular system of parameters for $\m$  in which  the ideal $I$  
is a monomial ideal.
\end{example}

\begin{remark}  
Another  description of the order valuation ring $V$ of $R_3$  in 
Example~\ref{not monomial} may be obtained as follows.  
Since  $R$ is equicharacteristic, the completion $\widehat R$ of $R$ has the form 
 $\widehat{R} = k [[x, y, z]]$,  where  $k$  is a field.
Let $u, w, t$ be indeterminates over $k$, and consider the $k[[x,y,z]]$-algebra 
homomorphism $\varphi : k[[x,y,z]]   \longrightarrow k(u,w)[[t]]$ 
obtained by mapping
	$$x \mapsto t^3, \quad y \mapsto t^3 + u t^4, \quad z \mapsto w t^5.$$
The map  $\varphi$ is an embedding and $V = k (u, w) [[ t ]] \cap \mathcal{Q} (R)$.
\end{remark}

Example \ref{7.11} illustrates a pattern where there are exactly 
two changes of direction from $R_0$ to $R_3$ and where $R_0/\m_0=R_3/\m_3$.

\begin{example}\label{7.11}
Let $(R,~ \m)$ be an equicharacteristic $4$-dimensional regular local ring.
Assume that $R = R_0 \subset R_1 \subset R_2 \subset R_3$ is a sequence of local quadratic transforms 
such that there is a change of direction from $R$ to $R_2$ and a change of direction from $R_1$ to $R_3$.
Theorem~\ref{is monomial2} implies that there exists a regular system of parameters  $x, y, z, w$ 
for $R$  such that the sequence from $R$ to $R_3$ is monomial with respect to these parameters. 
Moreover, we may assume the sequence of local quadratic transforms is one of 
the following two choices:  
$$
R:=R_0 ~\subset~  {^xR_1} ~\subset~ {^{yx}R_2} ~\subset ~{^{zyx}R_3}, \quad \text{or} \quad R:=R_0 ~\subset~  {^xR_1} ~\subset~ {^{yx}R_2} ~\subset ~{^{xyx}R_3}.
$$
Let $V_i$ denote the order valuation ring of $R_i$ with valuation $v_i$, for  $0 \le i \le 3$.

In the case of $R_0 \subset~ {^xR_1} \subset~{^{yx}R_2} \subset {^{zyx}R_3}$,      we have:
\begin{enumerate}
\item The valuations are defined by,
$$
\begin{array}{c|c|c|c|c }
				&x	&y	&z	&w	\\ \hline
v_3:=\ord_{R_3}		&4	&6	&7	&8	\\ \hline
v_2:=\ord_{R_2}		&2	&3	&4	&4	\\ \hline
v_1:=\ord_{R_1}		&1	&2	&2	&2	\\ \hline
v_0:=\ord_{R_0}		&1	&1	&1	&1
\end{array}
$$
\item $P_{R_0R_3}$ is given by
$$P_{R_0R_3} = (y, ~ z, ~ w)^{3} + (x^5, ~ x^3 y, ~ x^3 z, ~ x^3 w, ~ x^2 y^2, ~ x^2 y z, ~ x y w, ~ x^3 z^2, ~ x z w, ~ x w^2)$$
\item $\Rees P_{R_0R_3} = \{ V_0, V_1, V_3 \}$
\item $P_{R_0R_3} = \{ a \in \m ~|~ v_3 (\alpha) \ge 18 \}$.
\end{enumerate}

In the  case of  $R_0 \subset~ {^xR_1} \subset~{^{yx}R_2} \subset {^{xyx}R_3}$, , we have:
\begin{enumerate}
\item The valuations are defined by,
$$
\begin{array}{c|c|c|c|c }
				&x	&y	&z	&w	\\ \hline
v_3:=\ord_{R_3}		&3	&5	&7	&7	\\ \hline
v_2:=\ord_{R_2}		&2	&3	&4	&4	\\ \hline
v_1:=\ord_{R_1}		&1	&2	&2	&2	\\ \hline
v_0:=\ord_{R_0}		&1	&1	&1	&1
\end{array}
$$
\item $P_{R_0R_3}$ is given by
$$P_{R_0R_3} = (y, ~ z, ~ w)^{3} + (x^5, ~ x^4 y, ~ x^3 z, ~ x^3 w, ~ x^2 y^2, ~ x y z, ~ x y w, ~ x z^2, ~ x z w, ~ x w^2)$$
\item $\Rees P_{R_0R_3} = \{ V_0, V_1, V_3 \}$
\item $P_{R_0R_3} = \{ a \in \m ~|~ v_3 (\alpha) \ge 15 \}$.
\end{enumerate}

\end{example}

\end{document}